\documentclass[10pt,journal,compsoc]{IEEEtran}

\usepackage{amsfonts}
\usepackage{amsmath}
\usepackage{amsthm}
\usepackage[colorlinks=True,
            linkcolor=blue,
            anchorcolor=green,
            citecolor=blue
            ]{hyperref} 
\usepackage[noabbrev,capitalise,nameinlink]{cleveref}
\usepackage{float}
\usepackage{graphicx}
\usepackage{epstopdf}
\usepackage{multirow}
\usepackage{caption,subcaption}
 \usepackage{enumitem}
 \usepackage[ruled,vlined]{algorithm2e}
\usepackage{etoolbox}
 
\makeatletter
\patchcmd{\@algocf@start}% <cmd>
  {-1.5em}% <search>
  {0pt}% <replace>
  {}{}% <success><failure>
\makeatother

\usepackage{xpatch}

\makeatletter
\xpatchcmd\SetKwInOut
  {\hangafter=1\parbox[t]}
  {\hangafter=1\justify\parbox[t]}
  {}{\fail}
\makeatother
\usepackage{ragged2e}

\newtheorem{theorem}{Theorem}[section]

\newtheorem{lemma}{Lemma}[section]

\newtheorem{corollary}{Corollary}[section]
\newtheorem{definition}{Definition}[section]
\newtheorem{remark}{Remark}[section]
\newtheorem{assumption}{Assumption}[section]
\newtheorem{scheme}{Scheme}[section]
\newtheorem{example}{Example}[section]

\def\D{{\mathcal D}}
\def\K{{\mathcal K}} 

\def\L{{\mathcal L}}
\def\R{{\mathbb  R}}
\def\P{{\mathbb  P}}

\def\ba{\mathbf{a}}
\def\bb{\mathbf{b}}
 
\def\bw{\mathbf{w}}
\def\bv{\mathbf{v}} 
\def\bz{\mathbf{z}}
\def\tk{\tau_{k+1}}
\def\tauk{\Omega^{\tau_{k+1}}}
\def\blx{\overline{\mathbf w}} 
\def\bx{\mathbf{w}}

\def\bpi{{\boldsymbol  \pi}}

\def\bva{{\boldsymbol  \varphi}}
\def\bg{{\boldsymbol  g}}

\def\ADMM{{\tt FedADMM}}
\def\FA{{\tt FedAvg}}
\def\FP{{\tt FedProx}} 
 \def\FAT{{\tt FedAlt}}
\def\FS{{\tt FedSim}} 
\graphicspath{ {./images/} }

\ifCLASSINFOpdf
\else
\fi

\ifCLASSOPTIONcompsoc
\else
  \usepackage{cite}
\fi

\begin{document}
\title{Federated Learning via Inexact ADMM 
}

\author{
{Shenglong Zhou and Geoffrey Ye Li}% <-this % stops a space
\IEEEcompsocitemizethanks{
\IEEEcompsocthanksitem S.L. Zhou is with the School of Mathematics and Statistics, Beijing Jiaotong University, Beijing, China  (shlzhou@bjtu.ed.cn).
\IEEEcompsocthanksitem G. Y. Li is with the ITP Lab,  Department of Electrical and Electronic Engineering, Imperial College London, UK  (geoffrey.li@imperial.ac.uk). 

}

\thanks{~}}

\IEEEtitleabstractindextext{\justify
\begin{abstract} One of the crucial issues in federated learning is how to develop efficient optimization algorithms.   Most of the current ones require full device participation and/or impose strong assumptions for convergence.  Different from the widely-used gradient descent-based algorithms, in this paper,  we develop  an inexact alternating direction method of multipliers (ADMM), which is both computation- and communication-efficient,  capable of combating the stragglers' effect, and convergent under mild conditions. {Furthermore, it has a high numerical performance compared with several state-of-the-art algorithms for federated learning.}
\end{abstract}

% Note that keywords are not normally used for peer review papers.
}

\maketitle

% make the title area

% To allow for easy dual compilation without having to reenter the
% abstract/keywords data, the \IEEEtitleabstractindextext text will
% not be used in maketitle, but will appear (i.e., to be "transported")
% here as \IEEEdisplaynontitleabstractindextext when compsoc mode
% is not selected <OR> if conference mode is selected - because compsoc
% conference papers position the abstract like regular (non-compsoc)
% papers do!
\IEEEdisplaynontitleabstractindextext
% \IEEEdisplaynontitleabstractindextext has no effect when using
% compsoc under a non-conference mode.

% For peer review papers, you can put extra information on the cover
% page as needed:
% \ifCLASSOPTIONpeerreview
% \begin{center} \bfseries EDICS Category: 3-BBND \end{center}
% \fi
%
% For peerreview papers, this IEEEtran command inserts a page break and
% creates the second title. It will be ignored for other modes.
\IEEEpeerreviewmaketitle

\ifCLASSOPTIONcompsoc
\IEEEraisesectionheading{\section{Introduction} \label{sec:introduction}}
\else
\section{Introduction}
 \label{sec:introduction}
\fi
% Computer Society journal (but not conference!) papers do something unusual
% with the very first section heading (almost always called "Introduction").
% They place it ABOVE the main text! IEEEtran.cls does not automatically do
% this for you, but you can achieve this effect with the provided
% \IEEEraisesectionheading{} command. Note the need to keep any \label that
% is to refer to the section immediately after \section in the above as
% \IEEEraisesectionheading puts \section within a raised box.

% The very first letter is a 2 line initial drop letter followed
% by the rest of the first word in caps (small caps for compsoc).
%
% form to use if the first word consists of a single letter:
% \IEEEPARstart{A}{demo} file is ....
%
% form to use if you need the single drop letter followed by
% normal text (unknown if ever used by the IEEE):
% \IEEEPARstart{A}{}demo file is ....
%
% Some journals put the first two words in caps:
% \IEEEPARstart{T}{his demo} file is ....
%
% Here we have the typical use of a "T" for an initial drop letter
% and "HIS" in caps to complete the first word.
\IEEEPARstart{F}{ederated}  
learning (FL), originated from  \cite{konevcny2015federated, konevcny2016federated}, gains its popularity recently due to its ability to address various applications, such as vehicular communications \cite{samarakoon2019distributed, pokhrel2020federated,
elbir2020federated, posner2021federated}, digital health 
\cite{rieke2020future}, and mobile edge and over-the-air computing \cite{mao2017survey, yang2020federated, zhou2021communication, ye2022decentralized}. It is still facing many challenges. One of them is how to develop efficient optimization algorithms for different purposes, such as saving communication resources, accelerating the learning process, coping with the stragglers' effect,  just to name a few. We refer to some nice surveys \cite{kairouz2019advances,li2020federated,qin2021federated} for more challenges.
\subsection{Prior arts}\label{prior arts}
{Before  implementing FL into practical applications, many critical issues need to be addressed. We present a few of them that motivate our research in this paper.} 
\subsubsection{Communication efficiency}
{When exchanging parameters between clients and a central server, communication efficiency must be taken into account as frequent communications would consume expensive resources (e.g., transmission power, energy, and bandwidth). Popular techniques to improve communication efficiency include data compression and  reduction of communication rounds (CR). The former aims to quantize and sparsify the local parameters before the transmission so as to lessen the amount of the transmitted contents \cite{stich2018sparsified,di2018efficient,
wangni2018gradient,sattler2019robust}. For the latter, communications between the clients and the server occur in a periodic fashion so as to reduce CR
 \cite{mcmahan2017communication,li2019convergence,
 yu2019parallel,stich2018local, Lin2020Don}. In this paper, we will exploit this tactic.
}
\subsubsection{Computational efficiency}
{Since an increasing number of clients are engaging in the training, equipping all of them with strong computational capacities apparently is unrealistic. Hence, a desirable FL algorithm is able to reduce the computational complexity to alleviate clients' computational burdens. To achieve this goal, there are two promising solutions. The first one is the stochastic approximation of some critical items (e.g., the full gradients). This idea has been extensively exploited in the stochastic gradient descent (SGD) algorithms, such as the federated averaging ({\tt FedAvg} \cite{mcmahan2017communication}), local SGD  \cite{stich2018local, Lin2020Don}, and those in \cite{DeepLearning2015, AsynchronousStochastic2017, wang2021cooperative}. The second solution reduces the computational complexity by solving sub-problems inexactly, which has been widely adopted in the inexact ADMM. They allow clients to update their parameters via solving sub-problems approximately, thereby accelerating the computational speed exceptionally \cite{ding2019stochastic, Inexact-ADMM2021, ryu2022differentially, zhang2020fedpd}. We will take advantage of this technique in our algorithmic design.
 }

\subsubsection{Partial devices participation}
{Since the central server is unable to control the local devices and their communication environments, there may have delays/withdraw of sharing parameters by some clients due to inadequate transmission resources or limited computational capacity. This phenomenon is called the stragglers' effect,  namely, everyone waits for the slowest. A remedy for alleviating the effect is to let the central server pick  up a portion of clients in good conditions to take part in the training, which is known as the partial device participation  \cite{mcmahan2017communication,li2019convergence, li2020federatedprox}. Based on this scheme, FL algorithms can be categorized into two groups as follows.
}

{
 {a) \it Full device participation.}  There is an impressive body of work on developing algorithms based on full device participation, such as the non-stochastic gradient descent methods \cite{smith2018cocoa, LAG2018, wang2019adaptive, liu2021decentralized}, SGD \cite{yu2019parallel,stich2018local,  Lin2020Don, wang2021cooperative}, exact ADMM \cite{song2016fast, zhang2016dynamic, zhang2018improving, huang2019dp, elgabli2020fgadmm}, and inexact ADMM \cite{ding2019stochastic, Inexact-ADMM2021, ryu2022differentially, zhang2020fedpd}. However, due to the full device engagement, those algorithms  are at risk of the stragglers' effect, particularly in scenarios where large numbers of devices are distributed at the edge nodes. It is worth mentioning that ADMM-based methods have shown considerable popularity in solving the distributed optimization \cite{song2016fast,
zhang2018improving,huang2019dp,issaid2020communication,boyd2011distributed,zheng2018stackelberg, 
graf2019distributed}.}

% such as the non-statistical GD methods and SGD methods. The former construct the gradients using the full data \cite{smith2018cocoa, LAG2018, wang2019adaptive, liu2021decentralized}. By contrast, SGD methods take randomly chosen data to approximate the gradients. Therefore, they can accelerate the computational speed \cite{yu2019parallel,stich2018local,  Lin2020Don, wang2021cooperative} but usually assume  the local data is identically and independently distributed (i.i.d.) to establish the convergence, which is unrealistic for many FL applications. 
%
%In addition to GD-based frameworks, algorithms from primal-dual perspective have also drawn much attention. A popular representative is ADMM which has two versions: exact and inexact ADMM. The former requests clients to update their parameters through solving sub-problems exactly, which hence consumes expensive computational cost \cite{zhang2016dynamic, Li2017RobustFL,zhang2018improving, guo2018practical, zhang2018recycled, huang2019dp, elgabli2020fgadmm}. Inexact ADMM allows clients to  update their parameters via solving sub-problems approximately, thereby reducing the computational complexity \cite{ding2019stochastic, Inexact-ADMM2021, ryu2022differentially, zhang2020fedpd}.

 {
 {b) \it Partial device participation.} For realistic purpose, plentiful algorithms have been developed to choose partial devices for the training at every iteration, thereby enabling us to eliminate the stragglers' effect.  Popular representatives consist of {\tt FedAvg} \cite{mcmahan2017communication},  {\tt SCAFFOLD} \cite{karimireddy2020scaffold},  {\tt FedProx} \cite{li2020federatedprox}, {\tt FedAlt }\cite{collins2021exploiting, singhal2021federated, pillutla2022federated},  {\tt FedSim}  \cite{hanzely2021personalized, pillutla2022federated}, {\tt FedDCD} \cite{Fan2022}, and {\tt  FedSPD-DP} \cite{li2022federated}. The former five algorithms aim at solving the primal optimization problem while the latter two also investigate the dual problem.
 }
 
% \subsubsection{Convergence}
% We would like to point out that most of these algorithms not only require full device participation but also impose relatively strong assumptions on the model to establish the convergence properties. Common assumptions include the gradient Lipschitz continuity (also known as L-smoothness), strong smoothness, convexity, or strong convexity. 
% 
%  Motivated by this,    \FA, a state-of-the-art algorithm for FL, has been proposed in \cite{mcmahan2017communication}  and its convergence has been established under assumptions of strongly convexity and smoothness in \cite{li2019convergence}. Recently,  \FP\ has been developed in \cite{li2020federatedprox}  to tackle the statistical heterogeneity. We note that both algorithms were designed from the primal perspective. Hence,  it is interesting to see the performance of a primal-dual algorithm (e.g., ADMM) for FL using a partial device participation strategy.

\subsection{Our contributions}
The main contribution of this paper is to develop  an inexact ADMM-based FL algorithm (\ADMM, see Algorithm \ref{algorithm-ICEADMM}) with the following advantages. 

\begin{itemize}[leftmargin=17pt]
 \item[a)] 
{ \it Communication and computation efficiency.} The framework states the global averaging occurs only at certain steps (e.g., at steps $k$ being a multiple of a pre-defined integer $k_0$). This means  CR can be affected by setting a proper $k_0$.  It is shown that the larger $k_0$ the fewer CR for our algorithm to converge, see Figure \ref{fig:effect-k0}. In addition to the communication efficiency, \ADMM\ allows selected clients to solve their sub-problems approximately with a flexible accuracy.  In this regard, clients can relax the accuracy to lessen the computational complexity.

 \item[b)]
{ \it Eliminating the stragglers' effect.} In \ADMM, at each round of communication, the sever divides  all clients into two groups. One group adopts  the inexact ADMM to update their parameters, while parameters in the second group remain unchanged, which means the server can put stragglers into the second group to diminish their impact on the training.

 \item[c)]
{  \it Convergence under mild conditions.} It is noted that most of the aforementioned algorithms in Section 
\ref{prior arts} impose relatively strong assumptions on the model to establish the convergence. Common assumptions comprise the gradient Lipschitz continuity (also known as L-smoothness), convexity, or strong convexity. 
However, we have proven that \ADMM\ converges to a stationary point of the learning optimization problem  with a sub-linear convergence rate only under two mild conditions: gradient Lipschitz continuity and the coerciveness of the objective function, see Theorems  \ref{global-convergence-inexact} and \ref{complexity-thorem-gradient-inexact}. %If we further assume the convexity, then it can achieve the optimal parameter shown in Corollary \ref{L-global-convergence}. 

  \item[d)]
 {  \it High numerical performance.} The numerical comparison with several state-of-the-art algorithms has demonstrated that \ADMM\ can learn the parameter using the fewest CR and  the shortest computational time. 

\end{itemize}

\subsection{Organization}
The paper is organized as follows.  In the next section, we provide some mathematical preliminaries.  In  Section \ref{sec:ADMMFL},  we present algorithm \ADMM, followed by highlighting its advantages.  We  establish its global convergence and convergence rate in Section \ref{sec:convergence}. Numerical comparison and concluding remarks   are given in the last two sections.

\section{Preliminaries}
In this section, we present the notation to be employed throughout this paper and introduce ADMM and FL. 
\subsection{Notation}
We use  plain,  bold, and capital letters to present scalars, vectors, and matrices, respectively, e.g., $k$ and $\sigma$ are scalars, $\bx$ and $\bpi$  are vectors, $W$ and $\Pi$ are matrices.  {Let $\lceil t\rceil$ be the largest
integer smaller than $t+1$ (e.g.,  $\lceil 1.1\rceil= \lceil 2\rceil=2$).} Denote $[m]:=\{1,2,\cdots ,m\}$ with `$:=$' meaning define and  $\R^n$ the $n$-dimensional Euclidean space equipped with inner product $\langle\cdot,\cdot\rangle$ defined by $\langle\bx,\bz\rangle:=\sum_i w_{i}z_i$. The $2$-norm is written as $\|\cdot\|$, i.e.,  $\|\bx\|^2=\langle\bx,\bx\rangle$.  
 Function  $f$ is said to be  gradient Lipschitz continuous with a constant $r>0$ if 
 \begin{eqnarray}\label{Lip-r} 
\|\nabla f(\bx)-\nabla f(\bz  ) \|  \leq r\| \bx -\bz  \| 
  \end{eqnarray}
for any two vectors $\bx$ and $\bz$, where $\nabla  f(\bx)$  is the gradient of $f$ with respect to $\bx$. Hereafter, for two groups of vectors $\bx_i$ and $\bpi_i$ in $\R^n$, we denote 
\begin{eqnarray*}
\begin{array}{lll}
 W:=(\bx_1,\bx_2,\cdots ,\bx_m),~~ \Pi :=(\bpi_1,\bpi_2,\cdots ,\bpi_m).
\end{array}\end{eqnarray*}
Similar rules are also applied for $W^k, W^*, W^\infty$ and $\Pi^k, \Pi^*, \Pi^\infty$. Here $k, *$ and $\infty$ mean the iteration number, optimality and accumulation, e.g., see Corollary \ref{L-global-convergence}. 
 \subsection{ADMM}
We refer to the earliest work \cite{gabay1976dual} and a nice book \cite{boyd2011distributed} for more details of  ADMM and briefly introduce it as follows:  Given an optimization problem,
\begin{eqnarray*} %\label{ADMM-opt-prob}
\begin{array}{cll}
\min\limits_{\bx\in\R^n,\bz\in\R^q}~f (\bx)+  g(\bz),~
{\rm s.t.}~~ A\bx+B\bz-\bb=0, 
\end{array}\end{eqnarray*}
where $A\in\R^{p\times n}$,  $B\in\R^{p\times q}$, and   $\bb\in\R^{p}$, its corresponding augmented Lagrange function  is  
\begin{eqnarray*}%\label{ADMM-opt-ver11}
\arraycolsep=1.0pt\def\arraystretch{1.5}
\begin{array}{lll}
\L(\bx,\bz ,\bpi) &:=& f (\bx)+  g(\bz) \\
&+& \langle A\bx+B\bz-\bb, \bpi \rangle + \frac{\sigma}{2}\|A\bx+B\bz-\bb\|^2,
\end{array}\end{eqnarray*}
where $\bpi$ is the Lagrange multiplier and $\sigma$ is a given positive constant. Then starting with an initial point $(\bx^0, \bz^0,\bpi^0)$, ADMM performs the following steps  iteratively,
\begin{eqnarray*}\label{framework-ADMM-0}
 \arraycolsep=1.0pt\def\arraystretch{1.5}
\left\{\begin{array}{llll} 
 \bx^{k+1} &=&  {\rm argmin}_{\bx \in\R^n}~\L(\bx,\bz^k ,\bpi^k), \\ 
  \bz^{k+1} &=& {\rm argmin}_{\bz\in\R^q}~\L(\bx^{k+1},\bz,\bpi^k), \\
  \bpi^{k+1} &=&    \bpi^{k} + \sigma (A\bx^{k+1}+B\bz^{k+1}-\bb).
\end{array} \right.
\end{eqnarray*}

\subsection{Federated learning}
Suppose we have $m$ local clients (or devices) with datasets $\{\D_1,\D_2,\cdots ,\D_m\}$. Each client has the loss 
\begin{eqnarray*} 
 \arraycolsep=1.0pt\def\arraystretch{1.5}
\begin{array}{llll}
f_i(\bx) := \frac{1}{d_i}\sum_{{\bf x}\in\D_i} \ell_i(\bx; {\bf x}),
\end{array} 
\end{eqnarray*}
  where $\ell_i(\cdot; {\bf x}):\R^n\mapsto\R$ is a continuous loss function and bounded from below, $d_i$ is the cardinality of $\D_i$,  and $\bx\in\R^n$ is the parameter to be learned. 
%  Below are two examples  used for our numerical experiments.
%\begin{example}[Least square loss] \label{ex:lr} Suppose the $i$th client has data $\D_i=\{(\ba^i_1,b^i_1),\cdots ,(\ba^i_{d_i},b^i_{d_i})\}$, where $\ba^i_j\in\R^n$, $b^i_j\in\R$.  Then the least square loss is
%\begin{eqnarray*} 
% \arraycolsep=1.0pt\def\arraystretch{1.5}
%\begin{array}{llll}
%f_i(\bx)&=& \sum_{j=1}^{d_i} \frac{1}{2d_i}(\langle \ba^i_j,\bx\rangle-b^i_j)^2.
%\end{array} 
%\end{eqnarray*}
%\end{example}
%
%\begin{example}[$\ell_2$ norm regularized logistic loss]    \label{ex:lg}
%Similarly, the $i$th client has data $\D_i$ but with $b^i_j\in\{0,1\}$. The $2$-norm regularized logistic loss is given by 
%\begin{eqnarray*}  
% \arraycolsep=0pt\def\arraystretch{1.5}
%~~\begin{array}{llll}
%f_i(\bx)= 
%\frac{1}{d_i}\sum_{j=1}^{d_i}(\ln (1+{\rm e}^{\langle\ba^i_j,\bx\rangle} )-b^i_j\langle\ba^i_j,\bx\rangle)+\frac{\lambda}{2}\|\bx\|^2,
%\end{array} 
%\end{eqnarray*}
%where  $\lambda>0$ is a penalty parameter.
%\end{example}
The overall loss function  can be defined by 
\begin{eqnarray*}\begin{array}{lll}
f(\bx) :=   \sum_{i=1}^{m}  \alpha_{i} f_i(\bx),
\end{array}\end{eqnarray*}
where  $\alpha_{i}$ is a positive weight satisfying $\sum_{i=1}^{m}  \alpha_{i}=1.$  
 The task of FL is to learn  an optimal parameter $\bx^*$  that minimizes the overall loss, namely,
\begin{eqnarray}\label{FL-opt}
 \arraycolsep=1.0pt\def\arraystretch{1.5}
\begin{array}{lll}
 \bx^*:=\underset{\bx\in\R^n }{\rm argmin} ~f(\bx).
\end{array}\end{eqnarray}
Since $f_i$ is supposed to be bounded from below, we have
\begin{eqnarray}\label{FL-opt-lower-bound}\begin{array}{lll}
 f^*:=f(\bx^*)>-\infty.
\end{array}\end{eqnarray}
 
 \begin{algorithm*}[!th]
\SetAlgoLined
{\noindent \justifying Initialize an integer $k_0>0$ and $\Omega^0=[m]$. Set $k = 0$. Denote $\tau_k:=\lceil k/k_0 \rceil$ and $\bg_i^{k}:=\alpha_{i}\nabla f_i(\bx_i^{k}).$  All clients $i\in[m]$ initialize $\epsilon_i^0,\sigma_i>0, \nu_{i}\in[1/2,1), \bx_i^0$, $\bpi_i^0=-\bg_i^{0}$,   $\bz_i^0=\sigma_i\bx_i^0+\bpi_i^0$ and send the sever $\sigma_i$ to  calculate $\sigma=\sum_{i=1}^m\sigma_i$. }

% {\justifying \noindent  initializes $(\bx_i^0,\bpi_i^0)$, $\sigma _i>0$, and  $\bz^0_i =  \sigma _i \bx_i^0 + \bpi_i^0$.  Set a positive integer $k_0$ and $k = 0$. } 

\For{$k=0,1,2,\cdots $}{

\If{$  k\in\K:=\{0,k_0,2k_0,3k_0,\cdots \}$}{

{\justify\underline{\it{Weights upload: (Communication occurs)}}  
Clients in $\Omega^{\tau_k}$ send their parameters $\{\bz^{k}_i:i\in \Omega^{\tau_k}\}$  to the server. } \\
\vspace{1mm}

{\justifying \noindent \underline{\it{Global averaging:}} 
 The server calculates average parameter $\bx^{\tk}$ by} 
\begin{eqnarray}\label{iceadmm-sub1}
 \begin{array}{llll}
\bx^{\tk} =     \frac{1}{\sigma}\sum_{i=1}^m  \bz^{k}_i.
\end{array}
\end{eqnarray}
{\justifying \noindent \underline{\it Weights feedback: (Communication occurs)} 
 The  server randomly selects clients in $[m]$ to form a subset $\tauk$ and broadcasts them parameter $\bx^{\tk}$.}  
}

\For{every $i\in \tauk$}{
 {\justifying \noindent \underline{\it Local update:} Client $i$ updates its parameters  as follows:} 
\begin{eqnarray} 
\label{iceadmm-sub20}  
\hspace{-18mm} &&
\begin{array}{llll}
\epsilon_i^{k+1} ~\leq~ \nu_{i}\epsilon_i^{k}, 
    \end{array}\\[1.25ex] 
\label{iceadmm-sub2}  
\hspace{-18mm} &&
\begin{array}{llll}
\text{Find~} \bx^{k+1}_i \text{~such that}~ \|\bg_i^{k+1}  +  \bpi_i^{k}+\sigma_i(\bx_i^{k+1}-\bx^{\tk})\|^2\leq \epsilon_i^{k+1} ~ \text{by solving ${\rm min}_{\bx_i} L(\bx^{\tk},\bx_i, \bpi^k)$,} 
    \end{array}\\[1.25ex]  
\label{iceadmm-sub3}  
\hspace{-18mm}&& 
\begin{array}{llll}   
\bpi^{k+1}_i ~=~    \bpi_i^{k} +\sigma_i(\bx_i^{k+1}-\bx^{\tk}),  
\end{array}\\[1.25ex] 
\label{iceadmm-sub4} 
\hspace{-18mm}&& 
\begin{array}{llll}   
\bz^{k+1}_i ~=~   \sigma _i \bx_i^{k+1}  + \bpi_i^{k+1}. 
\end{array}
\end{eqnarray}
 }
 
 \For{every $i\notin \tauk$}{
 {\justifying \noindent \underline{\it Local invariance:} Client $i$ keeps their parameters by} 
\begin{eqnarray}
\label{iceadmm-sub5}
\begin{array}{llll}   
(\epsilon_i^{k+1}, \bx^{k+1}_i,\bpi^{k+1}_i,\bz^{k+1}_i) =(\epsilon_i^{k}, \bx^{k}_i,\bpi^{k}_i,\bz^{k}_i). 
\end{array}
\end{eqnarray}
 }
 
}
\caption{FL via inexact ADMM (\ADMM) \label{algorithm-ICEADMM}}
\end{algorithm*}

\section{FL via Inexact ADMM}\label{sec:ADMMFL}
 By introducing auxiliary variables, $\bx_i=\bx, i\in[m]$,  problem  \eqref{FL-opt} can be rewritten as the following form,
\begin{eqnarray}\label{FL-opt-ver1}\begin{array}{lll}
 \underset{ \bx, W}{\min}~~  \sum_{i=1}^{m}  \alpha_{i} f_i(\bx_i),~~{\rm s.t.}~~\bx_i=\bx,~i\in[m].
\end{array}\end{eqnarray}
Throughout the paper, we shall focus on the above optimization problem instead of problem  \eqref{FL-opt} as they are equivalent to each other. For simplicity, we also denote
\begin{eqnarray}\label{FX}
\begin{array}{lll}
F(W):=  \sum_{i=1}^{m} \alpha_{i} f_i(\bx_i).
\end{array}\end{eqnarray}
Clearly, $F(\bx,\bx,\cdots ,\bx)=f(\bx)$.

 \subsection{Algorithmic design}
 
  To implement ADMM for our problem \eqref{FL-opt-ver1},  
the corresponding augmented Lagrange function  can be defined by,  
\begin{eqnarray} \label{Def-L}
\arraycolsep=0pt\def\arraystretch{1.5}
\begin{array}{lll}
 \L(\bx,W,\Pi)&:=& \sum_{i=1}^m  L(\bx,\bx_i ,\bpi_i)\\
L(\bx,\bx_i ,\bpi_i)&:=& \alpha_{i} f_i(\bx_i){+}    \langle \bx_i{-}\bx, \bpi_i\rangle {+} \frac{\sigma_i}{2}\|\bx_i{-}\bx\|^2,
\end{array}\end{eqnarray}
where $ \Pi$ is the Lagrange multiplier, and $\sigma _i>0,i\in[m]$. The framework of ADMM for problem \eqref{FL-opt-ver1} is given as follows: For an initial  point  $(\bx^0, W^0, \Pi^0)$ and any $k\geq0$, perform the following updates iteratively,
\begin{eqnarray}\label{framework-ADMM-equ}
 \arraycolsep=1.0pt\def\arraystretch{1.5}
\left\{\begin{array}{llll} 
 \bx^{k+1} &=&   {\rm argmin}_{\bx}~\L(\bx,W^k, \Pi^k),\\ &=&\frac{1}{\sum_{i=1}^{m} \sigma _i}\sum_{i=1}^{m}  ( \sigma _i\bx^{k}_i+\bpi^k_i), \\ 
  \bx^{k+1}_i &=& {\rm argmin}_{\bx_i}L(\bx^{k+1},\bx_i, \bpi^k),~~& i\in[m], \\
  \bpi^{k+1}_i &=&    \bpi_i^{k} + \sigma _i(\bx_i^{k+1}-\bx^{k+1}),~~& i\in[m].
\end{array} \right.
\end{eqnarray} 
{To employ the above framework into FL, we treat $\bx^{k+1}$ as the global parameter updated by a central server and $(\bx^{k+1}_i, \bpi^{k+1}_i)$ as the local parameters updated by local client $i\in[m]$. However, this framework encounters several drawbacks. i) It repeats the three updates at every step, which means that the local clients and the central server have to communicate at every step, leading to communication inefficiency. ii) Solving the second sub-problem in  \eqref{framework-ADMM-equ} may incur an expensive computational cost as it generally does not admit a closed-form solution.   iii) Due to inadequate transmission resources or limited computational capacity, some clients may delay sharing parameters (i.e., stragglers' effect). So it is necessary to avoid selecting these clients. } 

{
To overcome these drawbacks, we cast a new algorithm into Algorithm  \ref{algorithm-ICEADMM} which aims at i) reducing communication rounds by averaging parameters only at certain steps (i.e., \eqref{iceadmm-sub1} occurs when $k\in\K$), ii) alleviating the computational burdens for clients by solving their sub-problems inexactly (i.e., computing \eqref{iceadmm-sub2}), and iii) diminishing the stragglers' effect by selecting a portion of clients (i.e., clients in $\tauk$) to join in the training at every step. More precisely, we have the following advantageous properties.}
%\section{Advantages of \ADMM}
\subsection{Communication efficiency}

Directly performing  framework \eqref{framework-ADMM-equ}  in an FL setting leads to the communication between local clients and the central server at every step, which would consume large amounts of communication resources (e.g., power and bandwidth).   Therefore, in Algorithm  \ref{algorithm-ICEADMM}, we allow a portion of clients (i.e.,  clients in $\tauk$) to update their parameters a few times (i.e., $k_0$ times) and then upload them to the central server. In other words, the central server collects parameters from local clients only at step $k\in\K=\{0,k_0,2k_0,3k_0,\cdots \}$. Here, choosing a proper $k_0$ can reduce CR significantly. It is worth mentioning that such an idea has been extensively employed in \cite{yu2019parallel,stich2018local, Lin2020Don, DeepLearning2015,AsynchronousStochastic2017,  wang2021cooperative}.

 \subsection{Fast computation} We emphasize that $\bx_i^{k+1}$ in \eqref{iceadmm-sub2}  is well defined. In fact, let $\bv_i^{*}$ be any optimal solution to  ${\rm min}_{\bx_i} L(\bx^{\tk},\bx_i, \bpi^k)$. Then it satisfies the following optimality condition,
 \begin{eqnarray}\label{update-w-*}
   \arraycolsep=1.0pt\def\arraystretch{1.5}
 \begin{array}{llll}
\alpha_{i}\nabla f_i(\bv_i^{*})  +  \bpi_i^{k}+\sigma_i(\bv_i^{*}-\bx^{\tk})=0.
 \end{array}
\end{eqnarray}  
 This means there always exists a point satisfying the condition in \eqref{iceadmm-sub2}.  Since ${\rm min}_{\bx_i} L(\bx^{\tk},\bx_i, \bpi^k)$ is an unconstrained optimization problem,  many solvers can be used to solve it. However, we are interested in algorithms that can find  $\bx_i^{k+1}$ quickly. Particularly, we initialize $\bv_i^0=\bx^{\tk}$, and perform the following steps, for $\ell=0,1,2,\cdots ,\kappa$,
\begin{eqnarray}\label{update-w-l}
   \arraycolsep=1.0pt\def\arraystretch{1.5}
 \begin{array}{llll}
&& \bv_i^{\ell+1}\\ 
  &=& {\rm argmin}_{\bx_i}~\langle \bx_i-\bx^{\tk}, \bpi_i^k\rangle + \frac{\sigma_i}{2}\|\bx_i-\bx^{\tk}\|^2+\\
 &&  \alpha_{i} (f_i(\bv _i^{\ell}) + \langle \nabla f_i(\bv _i^{\ell}), \bx_i- \bv _i^{\ell}  \rangle  + \frac{r_i}{2}\|\bx_i-\bv _i^{\ell}\|^2)\\
&=& \frac{1}{\alpha_{i}r_i   +\sigma_i}(\alpha_{i}r_i \bv_i^{\ell} +\sigma_i \bx^{\tk} - ( \alpha_{i}\nabla f_i(\bv_i^{\ell}) +  \bpi_i^{k})),
 \end{array}
\end{eqnarray}
where $r_i>0$ which can be set as the Lipschitz continuous constant if $f_i$ is Lipschitz continuous and $\kappa$ is a given maximum number of steps to update $\bv_i^{\ell}$.  The following theorem states that using (\ref{update-w-l})  to find $\bx_i^{k+1}{=}\bv_i^{\kappa+1}$ can guarantee the condition in (\ref{iceadmm-sub2}) within a small number of iterations $\kappa$. 
\begin{theorem}\label{finite-stop} Suppose that every $f_i, i\in[m]$ is gradient Lipschitz continuous with  $r_i>0$ and Hessian matrix $ \nabla^2 f_i \succeq -s_i I$ with $s_i\geq0$. By setting $\sigma_i\geq \alpha_{i} s_i+\varrho  \alpha_{i}r_i/2 $ with $\varrho>1$,  client $i\in[m]$ can find $\bx_i^{k+1}=\bv_i^{\kappa+1}$ such that (\ref{iceadmm-sub2}) through (\ref{update-w-l}) with at most $\kappa$ steps, where  
    \begin{eqnarray} \label{finite-stops}
   \begin{array}{lll}
\kappa = \log_{\varrho}\left \lceil \frac{2(\alpha_{i}^2r_i^2+\sigma_i^2)\|\bx^{\tk}-\bv_i^{*}\|^2}{\epsilon_i^{k+1}}  \right \rceil -1.
    \end{array} 
 \end{eqnarray}  
 \end{theorem}  
Here, $  \nabla^2 f_i \succeq -s_i I$  stands for $  \nabla^2 f_i + s_i I\succeq 0 $, a positive semi-definite matrix. All convex and plentiful non-convex  functions  satisfy this condition. For convex functions, we could choose $s_i=0$.  Based on \eqref{finite-stops}, clients can set a slightly large accuracy $\epsilon_i^{k+1}$ to ensure a small  $\kappa$  for the sake of  fastening their computation.  Therefore, the computational cost can be saved in comparison with solving the second sub-problem of \eqref{framework-ADMM-equ} exactly.

\subsection{Coping with straggler's effect}
 The framework  of \ADMM\ integrates partial device participation and hence can deal with the straggler's effect. {According to \cite{li2019convergence}, this can be done as follows: By setting a threshold $m_0\in[1,m)$, the   server collects the outputs of the first $m_0$ responded clients (to form $\tauk$). After collecting  $m_0$ outputs, the server stops waiting for the rest, namely the rest clients are deemed as stragglers in this iteration.} 
 
 Partial device participation has been exploited by \FA\ \cite{mcmahan2017communication},  \FP\ \cite{li2020federatedprox}, \FAT\ \cite{pillutla2022federated}, and so forth. Here, \FA\ presented in Algorithm \ref{algorithm-FA} corresponds to the case of $B=\infty$ in its original version. %For \FP, we use a single local update per device based on  gradient decent scheme \eqref{fedprox} to solve the sub-problem with a proximal term. 
 There are some difference among these algorithms and ours.  First of all,  the global averaging for \FP\ and \FAT\ is taken on the selected clients in $\Omega^{\tau_k}$, that is,
 \begin{eqnarray}\label{FP-sub1}
\arraycolsep=1.0pt\def\arraystretch{1.5} 
 \begin{array}{lcll}
\bx^{\tk} =     \frac{1}{|\Omega^{\tau_k}|}\sum_{i\in\Omega^{\tau_k} }  \bx^{k}_i,
\end{array}
\end{eqnarray}
while  \ADMM\ and \FA\ assemble parameters of all clients, see  \eqref{iceadmm-sub1} and \eqref{FA-sub1}. Moreover, the other three algorithms average parameters $\bx^{k}_i$ directly while \ADMM\ aggregates  $\bz^{k}_i$ which is a combination of  primal variable  $\bx^{k}_i$ and dual variable $\bpi^{k}_i$.  To this end, it is more secured to protect clients' data when communicating with the server.

\begin{algorithm} 
\SetAlgoLined
{\noindent \justifying Initialize an integer $k_0, \gamma>0$ and $\Omega^0=[m]$. Set $k = 0$. 
  All clients $i\in[m]$ initialize $ \bx_i^0=0$. }

% {\justifying \noindent  initializes $(\bx_i^0,\bpi_i^0)$, $\sigma _i>0$, and  $\bz^0_i =  \sigma _i \bx_i^0 + \bpi_i^0$.  Set a positive integer $k_0$ and $k = 0$. } 

\For{$k=0,1,2,\cdots $}{

\If{$  k\in\K:=\{0,k_0,2k_0,3k_0,\cdots \}$}{

\underline{\it{Weights upload: (Communication occurs)}}\\  
{\justify Clients in $\Omega^{\tau_k}$ send $\{\bx_i^{k}:i\in \Omega^{\tau_k}\}$  to the server. } \\
\vspace{2mm}

{\justifying \noindent \underline{\it{Global averaging:}}
 The server averages $\bx^{\tk}$ by} 
 %[\FA]
\begin{eqnarray}\label{FA-sub1}
\arraycolsep=1.0pt\def\arraystretch{1.5} 
 \begin{array}{lcll}
\bx^{\tk} =   \frac{1}{m}\sum_{i=1}^m  \bx^{k}_i.
\end{array}
\end{eqnarray}
 
\underline{\it Weights feedback: (Communication occurs)}

{\justifying \noindent The  server randomly selects clients to form $\tauk$ and broadcasts them $\bx^{\tk}$. } 
 }
 }

\For{every $i\in \tauk$}{
 {\justifying \noindent \underline{\it Local update:} Client $i$ updates its parameters by} 
\begin{eqnarray*} 
\label{fedvag}   
\arraycolsep=1pt\def\arraystretch{1.5}
\begin{array}{lcll}
\bx^{k+1}_i = \left\{ \begin{array}{llll}
\bx^{\tk} - \frac{\gamma}{m}    \nabla f_i(\bx^{\tk}),&~ k\in\K, \\
\bx_i^k-\frac{\gamma}{m}    \nabla f_i(\bx_i^k),&~ k\notin\K. 
    \end{array}\right.
 \end{array}   \end{eqnarray*}
 \For{every $i\notin \tauk$}{
 {\justifying \noindent \underline{\it Local invariance:} Client $i$ keeps $\bx^{k+1}_i=\bx^{k}_i$.} 
 }
 
}
\caption{\FA. \label{algorithm-FA}}
\end{algorithm}

\subsection{Local invariance} We would like to point out that clients outside $\Omega^{\tau_k}$ do nothing at steps $k,k+1,\cdots ,k+k_0-1$. We use \eqref{iceadmm-sub5} for the purpose of notational convenience when conducting convergence analysis. Moreover, \eqref{iceadmm-sub5} also allows the server to record the previous uploaded parameters from clients outside $\Omega^{\tau_k}$. Precisely, for any $i\notin\Omega^{\tau_k}$, at step $k\in\K$, let $k_i$ be the largest integer in $[0,k)$ such that $i\in\Omega^{\tau_{k_i}}$. In other words, $k_i$ is the last time that client $i$ was selected. Hence,  \eqref{iceadmm-sub5} implies $\bz_i^t\equiv \bz_i^{k_i}, \forall t=k_i,k_i+1,\cdots ,k$ and thus
\begin{eqnarray*}%\label{iceadmm-sub11}
   \arraycolsep=1.0pt\def\arraystretch{1.5}
 \qquad\begin{array}{llll}
\bx^{\tk} =   \frac{1}{\sigma}(\sum_{i\in \Omega^{\tau_k}} \bz^{k}_i + \sum_{i\notin \Omega^{\tau_k}} \bz^{k_i}_i) 
%&=& \frac{1}{\sigma}(\sum_{i\in \Omega^{\tau_k}} \bz^{k}_i + \sum_{i\notin \Omega^{\tau_k}} \bz^{k}_i)\\
=  \frac{1}{\sigma}\sum_{i=1}^m  \bz^{k}_i,
\end{array}
\end{eqnarray*}
which manifests that the sever uses the previously uploaded parameters (i.e., $\bz^{k_i}_i$) from  clients outside $\Omega^{\tau_k}$ and the currently uploaded parameters (i.e., $\bz^{k}_i$) from  clients in $\Omega^{\tau_k}$.
 
\section{Convergence analysis}\label{sec:convergence}
We aim  to establish the  global convergence and convergence rate  for \ADMM\ in this section, before which we define the optimality conditions of problems\  \eqref{FL-opt-ver1} and \eqref{FL-opt} as follows.  
\subsection{Stationary point} 
\begin{definition}\label{def-sta} A point $(\bx^*, W^*,\Pi^*)$ is a stationary point of   problem  (\ref{FL-opt-ver1}) if it satisfies 
\begin{eqnarray}\label{opt-con-FL-opt-ver1}
 \arraycolsep=1.0pt\def\arraystretch{1.5}
  \left\{\begin{array}{rcll}
 \alpha_{i}\nabla  f_i(\bx_i^*)+\bpi_i^* &=&   0, ~~&i\in[m],  \\ 
 \bx_i^*-\bx^* &=&0,&i\in[m],\\ 
  \sum_{i=1}^{m} \bpi_i^* &=& 0.
\end{array} \right.
\end{eqnarray} 
\end{definition}
It is not difficult to prove that any locally optimal solution to  problem  \eqref{FL-opt-ver1} must satisfy \eqref{opt-con-FL-opt-ver1}. If  $f_i$ is convex for every $i\in[m]$, then a point  is a globally optimal solution if and only if it satisfies condition \eqref{opt-con-FL-opt-ver1}.   Moreover,  a stationary point $(\bx^*, W^*,\Pi^*)$  of   problem  \eqref{FL-opt-ver1} indicates
\begin{eqnarray} \label{grad-x-*=0}
\begin{array}{llll}
\nabla   f(\bx^*) = \sum_{i=1}^{m}   \alpha_{i} \nabla   f_i(\bx^*) =  - \sum_{i=1}^{m}  \bpi_i^*=0.\end{array}
\end{eqnarray} 
That is, $\bx^*$ is also a stationary point of problem  \eqref{FL-opt}.  
\subsection{Some assumptions}
%Moreover, we need some assumptions on functions $f_i, i\in[m]$. 
\begin{assumption}\label{ass-fi} Every $f_i, i\in[m]$ is gradient Lipschitz continuous with a constant $r_i>0$. 
\end{assumption}
\begin{assumption}\label{ass-f} Function $f$ is coercive. That is, $ f(\bx)\rightarrow+\infty$ when $\|\bx\|\rightarrow +\infty$.    
\end{assumption}
\begin{scheme}\label{ass-omega} The sever randomly selects $\Omega^\tau$ that satisfies
\begin{eqnarray*}  
   \begin{array}{lll} \Omega^{\tau+1} \cup \Omega^{\tau+2}\cup\cdots \cup\Omega^{\tau+s_0}=[m],~~
   \forall \tau=0,s_0,2s_0,\cdots      \end{array}
 \end{eqnarray*} 
 where $s_0$ is a pre-defined postive integer.
\end{scheme}
Such a scheme indicates that for each group of $s_0$  sets $\{\Omega^{\tau+1}, \Omega^{\tau+2},\cdots,\Omega^{\tau+s_0}\}$, all clients should be chosen at least once. In other words, for any client $i\in[m]$, the maximum gap between its two consecutive selections is no more than $s_0$, namely,
\begin{eqnarray}  
\label{scheme-omega-2T}
   \begin{array}{lll} \max \left\{u-v: 
   \begin{array}{l}
   i\in\Omega^{v},  i\in\Omega^{u},\\[.5ex] 
   i\notin \Omega^{\tau}, \tau=v+1,\cdots ,u-1
   \end{array}
   \right\}\leq s_0.
   \end{array}
 \end{eqnarray} 
\begin{remark} \label{rem:omega} Scheme \ref{ass-omega} can be satisfied with a high probability. In fact, if $\Omega^1, \Omega^2,\cdots$ are selected independent and indices in $\Omega^{\tau}$ are uniformly sampled from $[m]$ without replacement, then the probability of client $i$ being selected in $\{\Omega^{\tau+1}, \Omega^{\tau+2},\cdots,\Omega^{\tau+s_0}\}$ is
\begin{eqnarray*} 
\arraycolsep=1.0pt\def\arraystretch{1.5}
\begin{array}{lrl}
p_i&:=&1-\P(i\notin\Omega^{\tau+1}, i\notin\Omega^{\tau+2},\cdots,i\notin\Omega^{\tau+s_0})\\
& =&1-\P(i\notin\Omega^{\tau+1})\P(i\notin\Omega^{\tau+2})\cdots\P(i\notin\Omega^{\tau+s_0})\\
&=&1-(1- \frac{|\Omega^{\tau+1}|}{m})(1- \frac{|\Omega^{\tau+2}|}{m})\cdots(1- \frac{|\Omega^{\tau+s_0}|}{m}),
\end{array} \end{eqnarray*}  
which tends to $1$. For example, $p_i=1-10^{-5}$ if $s_0=5$ and $|\Omega^{\tau}|=0.9m$ for any $ \tau\geq 1$.  
\end{remark}
\subsection{Global convergence} 
{The sketch of showing the convergence is as follows: by defining sequence $\{\widetilde\L^k\}$  as
\begin{eqnarray} 
 \label{def-c}   
 \arraycolsep=1.0pt\def\arraystretch{1.5}
 \begin{array}{lll} 
 \widetilde\L^k&:=&\L^k + \sum_{i=1}^m\frac{29 \epsilon_i^{k} }{(1-\nu_{i})\sigma_i},\\
 \L^{k}&:=&\L(\bx^{\tau_k},W^{k},\Pi^{k}),
    \end{array}
    \end{eqnarray} 
we first prove its decreasing property with a descent scale $ \sum_{i=1}^{m} \frac{\sigma _i}{10}(\|\bx^{\tk}{-}\bx^{\tau_k} \|^2 {+} \|\bx_i^{k+1}{-}\bx_i^{k}\|^2)$. It allows us to claim the convergence of $\{\widetilde\L^k\}$ and the vanishing of $\|\bx^{\tk}-\bx^{\tau_k} \|, \|\bx_i^{k+1}-\bx_i^{k}\|$, and $\|\bx_i^{k}-\bx^{\tau_k}\|$. Then these properties enable us to obtain   \eqref{L-local-convergence-limit-inexact} and \eqref{L-local-convergence-limit-grad-inexact}, which together with the optimality conditions shows the convergence of sequence $\{(\bx^{\tau_k},W^{k},\Pi^{k})\}$ itself.} Therefore, we first 
establish the decreasing property of   sequence $\{\widetilde\L^k\}$. 
\begin{lemma}\label{lemma-decreasing-1} Under Assumption \ref{ass-fi}, it holds that
\begin{eqnarray*}   
 \arraycolsep=1.0pt\def\arraystretch{1.5}
 \begin{array}{lll}
 ~~~~\widetilde\L^k{-} \widetilde\L^{k+1} { \geq}   \sum_{i=1}^{m} \frac{\sigma _i}{10}(\|\bx^{\tk}{-}\bx^{\tau_k} \|^2 {+}    \|\bx_i^{k+1}{-}\bx_i^{k}\|^2).
    \end{array}  
 \end{eqnarray*}    
    \end{lemma} 
The above result enables us to show the convergence of three sequences $\{f (\bx^{\tau_k})\}$, $ \{F(W^{k})\} $, and $\{\L^{k}\}$. 
\begin{theorem}\label{global-obj-convergence-inexact} Suppose that Assumptions \ref{ass-fi} and \ref{ass-f} hold. Every client $i\in[m]$ chooses $\sigma_i\geq 3\alpha_{i}r_i$ and the sever selects  $\Omega^{\tau_k}$ as Scheme \ref{ass-omega}.  Then the following results hold.
 \begin{itemize}[leftmargin=17pt]
  \item[a)]  Sequence $\{(\bx^{\tau_k},W^{k},\Pi^{k})\}$  is bounded.
 \item[b)] 
  Three  sequences $\{\L^{k}\}$, $\{ F(W^{k})  \}$, and $\{f (\bx^{\tau_k})\}$ converge to the same value, namely,
   \begin{eqnarray}  \label{L-local-convergence-limit-inexact}
   \begin{array}{lll}
  \lim\limits_{k\rightarrow \infty} \L^{k} =  \lim\limits_{k\rightarrow \infty} F(W^{k})  = \lim\limits_{k\rightarrow \infty} f(\bx^{\tau_k}).
    \end{array}
 \end{eqnarray} 
 \item[c)] $\nabla F(W^{k})$ and $\nabla f(\bx^{\tau_k})$ eventually vanish, namely,  
    \begin{eqnarray}  \label{L-local-convergence-limit-grad-inexact}
   \begin{array}{lll}
 \lim\limits_{k \rightarrow \infty}\nabla F(W^{k})  =\lim\limits _{k \rightarrow\infty} \nabla f(\bx^{\tau_k}) =0.
    \end{array} 
 \end{eqnarray} 
 \end{itemize}
 \end{theorem}  
 Theorem \ref{global-obj-convergence-inexact} establishes the convergence property of the objective function values.  In the following theorem, we would like to see the convergence
performance of sequence $\{(\bx^{\tau_k},W^{k},\Pi^{k})\}$  itself, which requires more conditions. 
%To proceed with that, we need the assumption on the boundedness of the following level set    \begin{eqnarray} \label{level-set-S} 
%  \begin{array}{l}
%  \S(\nu):=\{\bx\in\R^n: f(\bx)\leq \nu\}
%     \end{array}
%  \end{eqnarray} 
%   for a given $\nu>0$.  It is worth mentioning that the boundedness of the level set is frequently used to establish  the convergence properties of optimization algorithms. There are many functions satisfying this condition, such as the coercive functions\footnote{A continuous function $f:\R^n\mapsto \R$  is coercive if $ f(\bx)\rightarrow+\infty$ when $\|\bx\|\rightarrow +\infty$.}    
\begin{theorem}\label{global-convergence-inexact} Suppose that Assumptions \ref{ass-fi} and \ref{ass-f} hold. Every client $i\in[m]$ chooses $\sigma_i\geq 3\alpha_{i}r_i$ and the sever selects  $\Omega^{\tau_k}$ as Scheme \ref{ass-omega}.  Then the following results hold. 
 \begin{itemize}[leftmargin=17pt]
 \item[a)]  Any accumulating point $(\bx^{\infty},W^{\infty},\Pi^{\infty})$ of sequence $\{(\bx^{\tau_k},W^{k},\Pi^{k})\}$ is a stationary point of  problem (\ref{FL-opt-ver1}), where $\bx^{\infty}$ is a stationary point of problem (\ref{FL-opt}). 
 \item[b)] If further assuming that $\bx^{\infty}$ is isolated, then the whole sequence  converges to $(\bx^{\infty},W^{\infty},\Pi^{\infty})$. 
 \end{itemize}
 \end{theorem}  
We point out that the establishments of  Theorems \ref{global-obj-convergence-inexact} and \ref{global-convergence-inexact} do not rely on the choices of $\Omega^{\tau_k}$ explicitly due to Scheme \ref{ass-omega}. If the sever generates $\Omega^{\tau_k}$ randomly rather than using Scheme \ref{ass-omega}, then the above two theorems are valid with a high probability. In addition, since no convexity of $f_i$ or $f$ is imposed,   the sequence is guaranteed to converge to the stationary point of problems (\ref{FL-opt-ver1}) and (\ref{FL-opt}). If we have the convexity of $f$, then the sequence converges to the optimal solution to (\ref{FL-opt-ver1}) and (\ref{FL-opt}),   stated by the following corollary.
 
\begin{corollary}\label{L-global-convergence}  Suppose that Assumptions \ref{ass-fi} and \ref{ass-f} hold and $f$ is convex. Every client $i\in[m]$ chooses $\sigma_i\geq 3\alpha_{i}r_i$ and the sever selects  $\Omega^{\tau_k}$ as Scheme \ref{ass-omega}.  Then the following results hold. 
\begin{itemize}[leftmargin=12pt] 
\item[a)] Three  sequences converge to the optimal function value of problem (\ref{FL-opt}), namely,
 \begin{eqnarray}  \label{L-global-convergence-limit}
   \begin{array}{lll}
\lim\limits_{k\rightarrow \infty} \L^{k} =  \lim\limits_{k\rightarrow \infty} F(W^{k})  = \lim\limits_{k\rightarrow \infty} f(\bx^{\tau_k})=  f^*.
    \end{array} 
 \end{eqnarray}   
  \item[b)] Any accumulating point  $(\bx^{\infty},W^{\infty},\Pi^{\infty})$ of   sequence  $\{(\bx^{\tau_k},W^{k},\Pi^{k})\}$ is an optimal solution to problem (\ref{FL-opt-ver1}), where $\bx^{\infty}$ is an optimal solution to  problem  (\ref{FL-opt}). 
 
\item [c)]  If  $f$ is strongly convex, then whole sequence  converges the unique optimal solution $(\bx^*,W^*,\Pi^*)$ to   problem (\ref{FL-opt-ver1}), where $\bx^*$ is the unique optimal solution to   problem (\ref{FL-opt}).  
\end{itemize} 
\end{corollary}

  \begin{remark}Regarding assumptions in  Corollary \ref{L-global-convergence},  $f$ being strongly convex does not require that every $f_i,i\in[m]$ is strongly convex. If one of $f_i$s is strongly convex and the remaining is convex, then $f=\sum_{i=1}^{m} \alpha_{i}f_i$ is strongly convex.  Moreover, the strongly convexity suffices to the coerciveness of $f$. Therefore, under the strongly convexity, Assumption \ref{ass-f} can be exempted.
  \end{remark}

 \subsection{Convergence rate}
{We have shown that   Algorithm \ref{algorithm-ICEADMM} converges. Now, we would like to see how fast this convergence is, stated as follows. 
 \begin{theorem} \label{complexity-thorem-gradient-inexact} Suppose that Assumptions \ref{ass-fi} holds. Every client chooses $\sigma_i\geq 3\alpha_{i}r_i,i\in[m]$ and the sever selects   $\Omega^{\tau_k}$ as Scheme \ref{ass-omega}.  Then for any $k>s_0k_0$, it has 
\begin{eqnarray}\label{error-minimal-grad-s} 
   \arraycolsep=1.0pt\def\arraystretch{1.5}
\begin{array}{rcll}
\min_{s=1,2,\cdots,k}\|\nabla f(\bx^{\tau_{s +1}})\|^2  \leq    \frac{c k_0}{k-s_0k_0}, 
\end{array} \end{eqnarray}
  where $c:=\frac{940s_0(s_0+1)m\max_{i\in[m]}\sigma_i(\widetilde\L^{0} - f^*) }{9}   +\sum_{i=1}^m \frac{169ms_0\epsilon_i}{1-\nu_i}$.
 \end{theorem}
According to the above theorem, the minimal value among  $\{\|\nabla f (\bx^{\tau_{s +1}})\|^2,s{\in}[k]\}$ vanishes with a  rate $O(1/k)$, a sub-linear rate. We emphasize that the establishment of such a convergence rate requires nothing but the assumption of gradient Lipschitz continuity, namely, Assumption \ref{ass-fi}. Similar results can be found in many literature. For example, in \cite{karimireddy2020scaffold} the convergence rate is about $O(\sqrt{1/k})$ while the rate in \cite{ zhang2020fedpd,zhou2022} is about $O(1/k)$ but has been obtained under the full device participation (corresponding to the case of $s_0=1$ in Scheme \ref{ass-omega}).
\begin{remark}\label{remark-com} Theorem \ref{complexity-thorem-gradient-inexact} suggests that Algorithm \ref{algorithm-ICEADMM} should be terminated if the following condition is satisfied,  
       \begin{eqnarray}\label{stopping-tol}
  \arraycolsep=0pt\def\arraystretch{1.5}
  \begin{array}{lllll}
\| \nabla f (\bx^{\tau_{k+1}})\|^2 \leq   \varepsilon,
   \end{array}
  \end{eqnarray} 
  where $\varepsilon$ is a given tolerance. Based on (\ref{error-minimal-grad-s}),  after 
         \begin{eqnarray}\label{stopping-iter}
  \arraycolsep=0pt\def\arraystretch{1.5}
  \begin{array}{lllll}
k =   \left\lceil \frac{(c+\varepsilon s_0)  k_0 }{\varepsilon} \right\rceil
   \end{array}
  \end{eqnarray} 
  iterations,  Algorithm \ref{algorithm-ICEADMM} meets  (\ref{stopping-tol}) and the total CR is
         \begin{eqnarray}\label{stopping-iter}
  \arraycolsep=0pt\def\arraystretch{1.5}
  \begin{array}{lllll}
CR:=\left\lceil \frac{2k}{k_0} \right\rceil  = \left\lceil \frac{2(c+\varepsilon s_0) }{\varepsilon} \right\rceil.
   \end{array}
  \end{eqnarray} 
  The above relation implies that the larger $s_0$ the more CR required by Algorithm \ref{algorithm-ICEADMM} to converge, which is reasonable. In fact,  one can observe that Scheme \ref{ass-omega} can be satisfied with a larger $s_0$. This is because,  a larger $s_0$ allows us to choose fewer clients to form $\Omega^\tau$, namely, fewer clients participating in the training at every step, which apparently results in slower convergence. As a consequence, the algorithm needs higher CR, thereby wasting communication resources. Hence, in order to reduce CR, it is essential to set an appropriately small $s_0$. However, to meet Scheme \ref{ass-omega}, a small $s_0$ means that more clients take part in the training, which will incur higher computational complexity.
 \end{remark}}
 
% \begin{remark} Theorem \ref{complexity-thorem-gradient-inexact} provides us a hint to terminate Algorithm \ref{algorithm-ICEADMM}. That is, 
%       \begin{eqnarray}\label{stopping-tol}
%  \arraycolsep=0pt\def\arraystretch{1.5}
%  \begin{array}{lllll}
%\| \nabla f (\bx^{\tau_k})\|^2 \leq   \epsilon,
%   \end{array}
%  \end{eqnarray} 
%  where $\epsilon$ is a given tolerance. Based on Theorem \ref{complexity-thorem-gradient-inexact},  after 
%         \begin{eqnarray}\label{stopping-iter}
%  \arraycolsep=0pt\def\arraystretch{1.5}
%  \begin{array}{lllll}
%k = \left\lceil \frac{\rho k_0(\varphi^1 -f^*) }{\epsilon} \right\rceil,
%   \end{array}
%  \end{eqnarray} 
%  iterations,  our method meets stopping criterion (\ref{stopping-tol}), where $\lceil  t\rceil$ refers to the smallest integer greater than $t$. And the total communication rounds (CR) are
%         \begin{eqnarray}\label{stopping-iter}
%  \arraycolsep=0pt\def\arraystretch{1.5}
%  \begin{array}{lllll}
%CR:=\left\lceil \frac{k}{k_0} \right\rceil  = \left\lceil \frac{\rho(\varphi^1 -f^*) }{\epsilon}  \right\rceil,
%   \end{array}
%  \end{eqnarray}  
% \end{remark}

\section{Numerical Experiments}\label{sec:num}
In this section, we conduct some numerical experiments to demonstrate the performance of \ADMM\ (available at \url{https://github.com/ShenglongZhou/FedADMM}). All numerical experiments are implemented through MATLAB (R2019a) on a laptop with 32GB memory and 2.3Ghz CPU.  

 \subsection{Testing example}
% We use Example  \ref{ex:lr} with synthetic data and Example \ref{ex:lg} with real data to conduct the numerical experiments.% Both objective functions are  gradient Lipschitz continuous.
 \begin{example}[Linear regression with non-i.i.d. data]\label{ex-linear} For this problem, local clients have their objective functions as
 \begin{eqnarray*} 
 \arraycolsep=1.0pt\def\arraystretch{1.5}
\begin{array}{llll}
f_i(\bx)&=& \sum_{t=1}^{d_i} \frac{1}{2d_i}(\langle \ba_i^t,\bx\rangle-b_i^t)^2,
\end{array} 
\end{eqnarray*}
where $\ba_i^t\in\R^{n}$ and $b_i^t\in\R$ are the $t$-th sample for client $i$. We first pick $m$ integers $d_1,\cdots,d_m$ randomly from  $[50,150]$ and denote $d:=d_1+\cdots+d_m$. Then we generate $\lceil d/3 \rceil$ samples $(\ba, b)$  from the standard normal distribution, $\lceil d/3 \rceil$ samples from the Student's $t$ distribution with degree $5$, and $d-2\lceil d/3 \rceil$ samples from the uniform distribution in $[-5,5]$. Now we shuffle all samples and divide them into $m$ parts with sizes $d_1\cdots,d_m$ for $m$ clients. In the regard, each client has non-i.i.d. data.   % For simplicity, we fix $n=100$ and alter $m\in\{32,64,96,128,160\}$. 
\end{example}
\begin{example}[Logistic regression]\label{ex-logist} For this problem,  local clients have their objective functions as 
\begin{eqnarray*}  
 \arraycolsep=0pt\def\arraystretch{1.5}
~~\begin{array}{llll}
f_i(\bx)= 
\frac{1}{d_i}\sum_{t=1}^{d_i}(\ln (1+{\rm e}^{\langle\ba_i^t,\bx\rangle} )-b_i^t\langle\ba_i^t,\bx\rangle)+\frac{\lambda}{2}\|\bx\|^2,
\end{array} 
\end{eqnarray*}
where $\ba_i^t\in\R^{n}, b^t_i\in\{0,1\}$ , and $\lambda>0$ is a penalty parameter (e.g., $\lambda=0.001$ in our  numerical experiments).  We use two real datasets described in Table \ref{tab:datasets} to generate $(\ba, b)$. Again, we randomly split $d$ samples into $m$ groups for $m$ clients.% where $m\in\{32,64,128\}$.% It has shown in \cite[Lemma 4]{wang2019greedy} that $f_i$ defined by (\ref{logist-loss}) is the gradient Lipschitz continuous with a constant $r_i=\lambda_{\max}(A_i^\top A_i)/4+\mu$, where $A_i$ is given similarly to Example \ref{ex-linear}.
\end{example}
\begin{table}[!th]
	\renewcommand{\arraystretch}{1.25}\addtolength{\tabcolsep}{-2pt}
	\caption{Descriptions of  two real datasets.}\vspace{-3mm}
	\label{tab:datasets}
	\begin{center}
		\begin{tabular}{lllrrrr }
			\hline
Data&Datasets&	Source	&	$n$	&	$d$\\\hline
%\texttt{gis} & Gisette& libsvm& 5000& 6000\\
\texttt{qot}&	Qsar oral toxicity	&	UCI &	1024 	&	8992 	\\
% \texttt{sct}&	Santander customer transaction	&	Kaggle	&	200 	&	200000 	\\
 { \texttt{rls}}&	{real-sim}	 &{Libsvm}&	{20958}&	{72309}\\

% {\texttt{higg}}&	 {higgs}	&	 {uci}	&	 {28} 	&	 {$11,000,000$} 	\\
% \texttt{rtb}&	Real time bidding	&	kaggle	&	88 	&	 {1000000} 	\\
\hline
 		\end{tabular}
	\end{center}
	\vspace{-5mm}
\end{table}

\subsection{Implementations}
 We fix $\alpha_{i}=1/m, i\in[m]$ in model \eqref{FL-opt} and initialize $\bx_i^0=\bpi_i^0=0$. Parameters are set as follows:  for each $i\in[m]$, let $\epsilon_i^0= k_0^2$, and $\nu_{i}=0.95$, and $\sigma_i=0.2r_i/m$, where $r_i$ is the gradient Lipschitz continuous constant for $f_i$. We terminate  our algorithm if  the following condition  is satisfied,
% \begin{eqnarray} \label{stopping}
% \arraycolsep=1.0pt\def\arraystretch{1.5}
%\begin{array}{rrr}
%{\rm CR}  &>& 1000,\\ 
% \|\nabla f(\bx^{\tau_{k}})\|^2
%& <& \frac{\varepsilon n}{md},\\
% \frac{{\rm std}\{ f(\bx^{\tau_{k}-3}),f(\bx^{\tau_{k}-2}) ,f(\bx^{\tau_{k}-1}),f(\bx^{\tau_{k}})\}}{1+|f(\bx^{\tau_{k}})|}  &<&   10^{-5},
%\end{array} 
%\end{eqnarray}
 \begin{eqnarray} \label{stopping}
 \arraycolsep=1.0pt\def\arraystretch{1.5}
\begin{array}{rrr} 
 \|\nabla f(\bx^{\tau_{k}})\|^2
& <& \min\left\{\frac{1}{5}\|\nabla f(0)\|^2,\frac{5\varepsilon n}{md}\right\}, 
\end{array} 
\end{eqnarray}
where $\varepsilon=10^{-3}$ for Example \ref{ex-linear} and $\varepsilon=10^{-7}$ for Example \ref{ex-logist}.  In the subsequent numerical experiments,  instead of using Scheme \ref{ass-omega}, we generate $\Omega^\tau$ randomly since it is  easier than Scheme \ref{ass-omega}. As mentioned in Remark \ref{rem:omega},   Scheme \ref{ass-omega} can be guaranteed with a high probability in this way. Specifically, let $\Omega^1, \Omega^2,\cdots$ be selected independently with $|\Omega^\tau|=\rho m$ for any $\tau\geq1$, where $\rho\in(0,1]$. Indices in each $\Omega^\tau$ are uniformly sampled from $[m]$ without replacement.

%\begin{table}[H]
%	\renewcommand{\arraystretch}{1.25}\addtolength{\tabcolsep}{1pt}
%	\caption{Choices of $t$ and $H_i$.}\vspace{-3mm}
%	\label{tab:choice-Hi}
%	\begin{center}
%		\begin{tabular}{llrr } \hline
%  &  &	 	&	  \\ 
%   &$t$ &	 $H_i$ &	 $H_i$ \\\hline
%Example \ref{ex-linear}&$0.15$&	$\frac{1}{d_i}A_i^\top A_i$&	$\frac{1}{d_i}\|A_i^\top A_i\| $	\\
%Example \ref{ex-logist}&$\max\{0.1,\frac{8}{n}{\rm ln}(d)\}$	&	$\frac{1}{4d_i}A_i^\top A_i $&	$\frac{1}{4d_i}\|A_i^\top A_i\|$\\
%\hline
% 		\end{tabular}
%	\end{center} \vspace{-5mm}
%\end{table}  

\subsection{Benchmark algorithms}
{We will compare \ADMM\ with \FA\ \cite{mcmahan2017communication} described in Algorithm \ref{algorithm-FA},  \FP\ \cite{li2020federatedprox}, and {\tt FedAlt}  and {\tt FedSim}  \cite{pillutla2022federated}. For \FP,  every selected client $i\in\Omega^{\tau_{k+1}}$ needs to approximately solve a sub-problem  at each iteration. We adopt the gradient descent method to  tackle it using an initial point  $\bx^{\tk}$ if $ k\in\K$ and $\bx_i^k$ if $k\notin\K$.   For {\tt FedAlt}  and {\tt FedSim}, we also employ the strategy that the global averaging only occurs at $k\in\K$ and exploit a partial model personalization $h_i$ from \cite{hanzely2021personalized}, i.e., 
    \begin{eqnarray*}
\arraycolsep=0pt\def\arraystretch{1.5}
\begin{array}{lcll}
h_i(\bx,\bv)=(1-\alpha) f_i(\bx) + \alpha f_i(\bv)+\frac{\mu}{2}\|\bx-\bv\|^2,
\end{array}
\end{eqnarray*}
where $\alpha=0.5$ and $\mu=0.001$ are used in the numerical experiments. %Moreover,   the learning rate for the four algorithms is set as $\gamma{=}\gamma_k(a){:=}a/{\rm log_2}(k+1)$ with $a=0.1/k_0$ for Example \ref{ex-linear} and $a=2d/m$ for Example \ref{ex-logist}. 
To ensure fair comparison, we initialize all algorithms with $\bx^0=\bx_i^0=0,i\in[m]$. In addition, we first implement \ADMM\ to solve the problem and terminate it if its solution  $\bw^{\tau_k}$ satisfies condition  \eqref{stopping}.  Then we  employ the other four algorithms to solve the problem if its solution $\bw$ meets the following condition,
    \begin{eqnarray*}
\arraycolsep=0pt\def\arraystretch{1.5}
\begin{array}{lcll}
f(\bw)-f(\bw^{\tau_k})\leq 2(1+|f(\bw^{\tau_k})|)10^{-4}.
\end{array}
\end{eqnarray*}
This condition allows all algorithms to stop with producing similar objective function values.
   }

\subsection{Numerical comparisons}
We compare five algorithms by reporting the following factors:  objective function values $f(\bx^{\tau_k})$, CR, and computational time (in seconds). {It is noted that there are four influential parameters $(n,m,\rho,k_0)$, where $n$ is the dimension of the solution, $m$ is the number of clients, $k_0$ has the impact on the CR, and $\rho\in(0,1]$ is the participation rate (i.e., the bigger $\rho$ the more clients to be chosen for the training at every iteration). To see the effect of one parameter, we will fix the others in the sequel.}

\subsubsection{Effect of $k_0$} 
{To see this, we fix  $(n,m,\rho){=}(100, 100,0.5)$. Here,  $\rho=0.5$ means half clients chosen for the training (i.e., $|\Omega^\tau|=0.5m$). First, we perform five algorithms to solve  Example \ref{ex-linear}  under $k_0\in\{1,10,30,50\}$ and report the results in Figure  \ref{fig:iterations-grad}.  One can observe that with the increasing of CR, all algorithms eventually achieve the same objective function value (i.e.,  the optimal one). Because of this, we will not report the objective function values in the subsequent numerical comparison. When $k_0{=}1$, the objective function values obtained by \ADMM\ decline slowly at the first several steps but approach the optimal value quickly afterwards. When $k_0{>}1$, it always outperforms the others as it uses the lowest CR.}
\begin{figure}[!th]
	\begin{subfigure}{.245\textwidth}
	\centering
	\includegraphics[width=.995\linewidth]{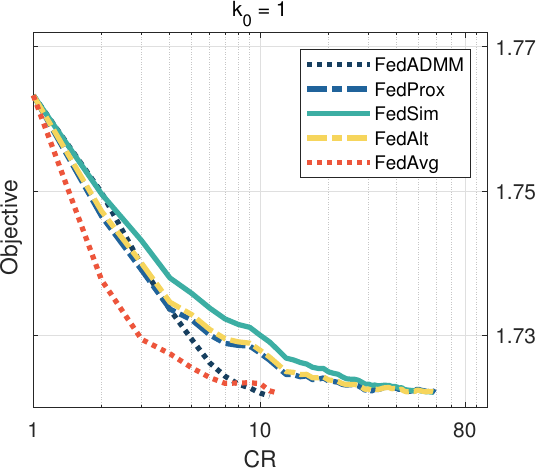}
\end{subfigure}	
\begin{subfigure}{.245\textwidth}
	\centering
	\includegraphics[width=.995\linewidth]{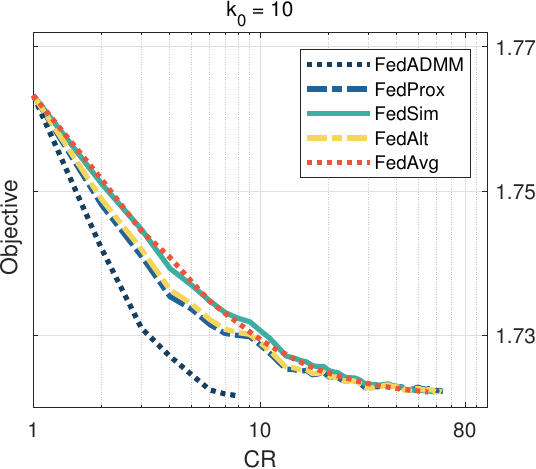}
\end{subfigure} 
~\vspace{1mm}\\
\begin{subfigure}{.245\textwidth}
	\centering
	\includegraphics[width=.995\linewidth]{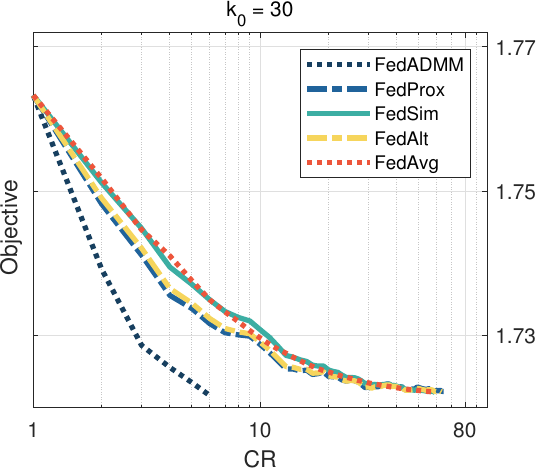}
\end{subfigure}  
\begin{subfigure}{.245\textwidth}
	\centering
	\includegraphics[width=.995\linewidth]{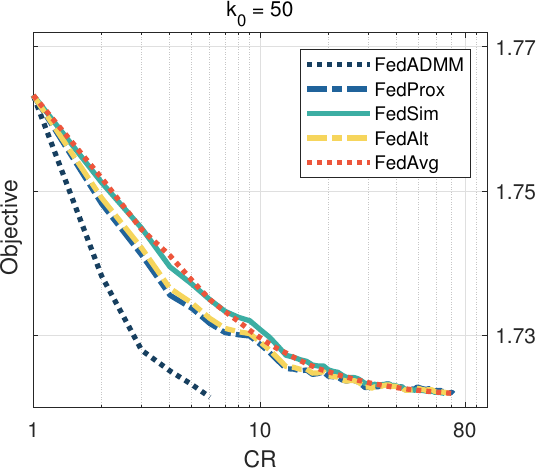}
\end{subfigure}  \vspace{-2mm}
\caption{Objective v.s. CR for Example \ref{ex-linear}.\label{fig:iterations-grad}}
 
\end{figure}

{We next generate 20 instances and report the results in terms of the median values in Figure  \ref{fig:effect-k0}, where each instance is solved by one algorithm under different choices of $k_0\in\{1,5,10,\cdots,50\}$. For example, when $k_0=10$, \ADMM\ solves the 20 instances and obtains 20 values of CR. The data reported in the figure is the median of these 20 values. Based on the results presented in Figure  \ref{fig:effect-k0},  we have the following comments. For Example \ref{ex-linear}, when $k_0$ is increasing, there is a descending trend of CR for \ADMM\ but an ascending trend for the other four algorithms. However, for Example \ref{ex-logist}, CR generated by every algorithm is declining with the rising of $k_0$. Apparently, for both examples the larger $k_0$ the longer the computational time and \ADMM\ behaves the best in terms of using the fewest CR and running the fastest.}

\begin{figure}[!th]
	\begin{subfigure}{.245\textwidth}
	\centering
	\includegraphics[width=.995\linewidth]{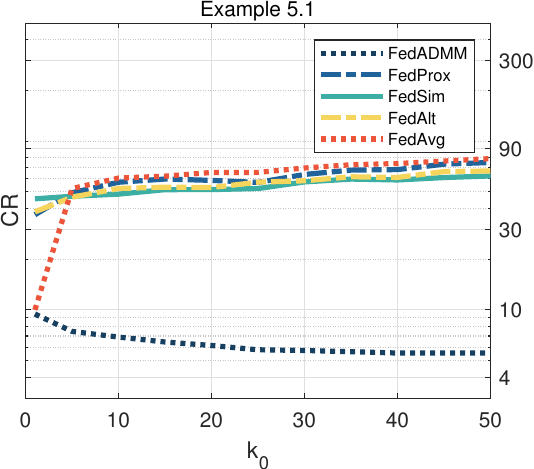}
\end{subfigure}	
\begin{subfigure}{.245\textwidth}
	\centering
	\includegraphics[width=.995\linewidth]{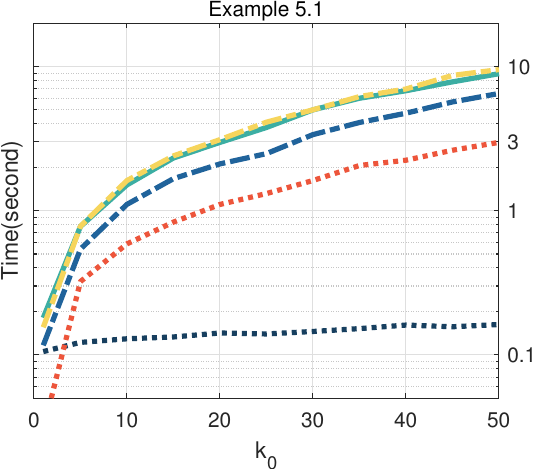}
\end{subfigure} 
~\vspace{1mm}\\ 
	\begin{subfigure}{.245\textwidth}
	\centering
	\includegraphics[width=.995\linewidth]{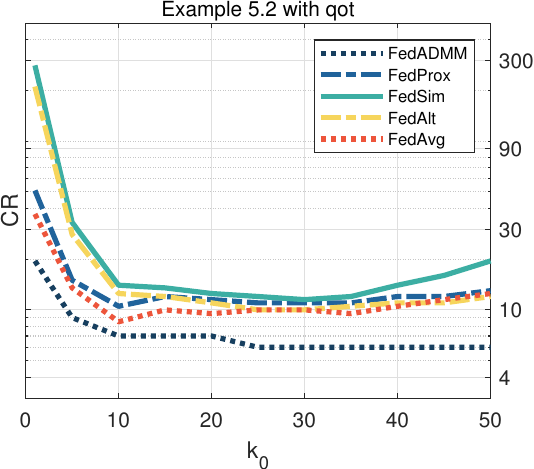}
\end{subfigure}	
\begin{subfigure}{.245\textwidth}
	\centering
	\includegraphics[width=.995\linewidth]{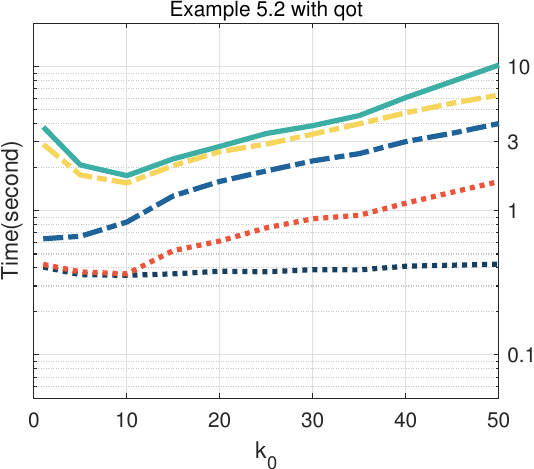}
\end{subfigure} \vspace{-5mm}
\caption{Effect of $k_0$.\label{fig:effect-k0}}
 
\end{figure}

\subsubsection{Effect of participation rate $\rho$}  { To see the effect of participation rate $\rho$  on the performance of each algorithm, we fix $(n,m,k_0){=}(100,100,10)$ and alter $\rho{\in}\{0.1,0.2,\cdots,0.9\}$.
Similarly, we report the median values over 20 instances in Figure  \ref{fig:effect-rho}. First, for Example \ref{ex-linear}, when $\rho$ is getting bigger (i.e., more and more clients are selected for the training), as expected that every algorithm consumes fewer CR, which results in shorter computational time. However, the picture for Example \ref{ex-logist} is slightly different. From the figure, when $\rho$ is varying, CR stabilizes at a certain level for \ADMM\ while sightly fluctuating for the other four algorithms. Basically, the higher value of $\rho$ the longer computational time spent by every algorithm for this example. Once again,  \ADMM\ outperforms the other algorithms for most scenarios.
}

\begin{figure}[H]
	\begin{subfigure}{.245\textwidth}
	\centering
	\includegraphics[width=.995\linewidth]{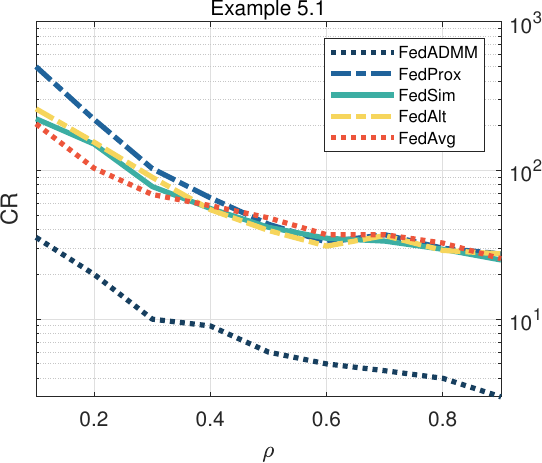}
\end{subfigure}	
\begin{subfigure}{.245\textwidth}
	\centering
	\includegraphics[width=.995\linewidth]{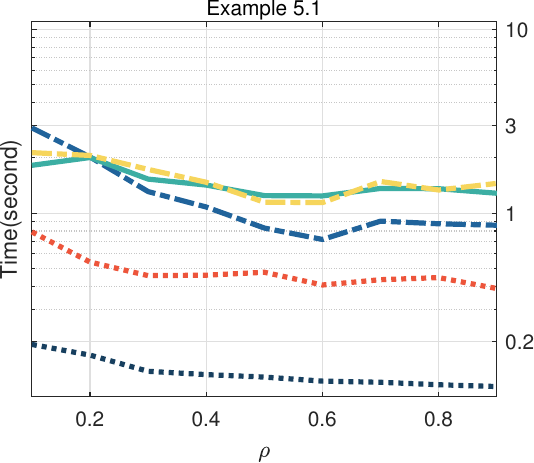}
\end{subfigure} 
~\vspace{1mm}\\ 
	\begin{subfigure}{.245\textwidth}
	\centering
	\includegraphics[width=.995\linewidth]{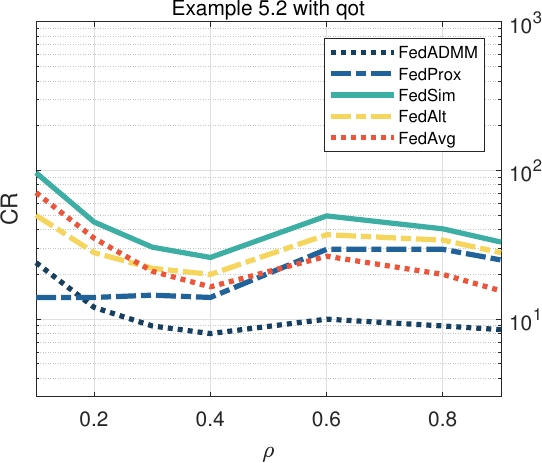}
\end{subfigure}	
\begin{subfigure}{.245\textwidth}
	\centering
	\includegraphics[width=.995\linewidth]{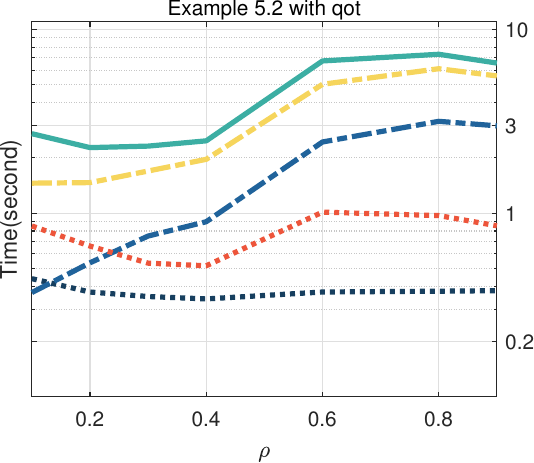}
\end{subfigure} \vspace{-2mm}
\caption{Effect of $\rho$.\label{fig:effect-rho}}
\end{figure}
\subsubsection{Effect of $m$} {Similarly, we fix $(n,\rho,k_0)=(100,0.5,10)$ but alter $m\in\{50,100,150,200\}$.   The median results over 20 instances are presented in Figure  \ref{fig:effect-m}.  For Example \ref{ex-linear}, it can be clearly seen that the bigger $m$ (i.e. the more clients joining in the training) the fewer CR but the longer computational time. This is reasonable as every algorithm needs to perform more computation. By contrast, for Example \ref{ex-logist}, there is a declining trend of CR for \ADMM\ along with $m$ increasing. However, CR used the other algorithms is descending when $m\in\{50,100\}$ and ascending afterwords.   Once again,  \ADMM\ achieves the best results, i.e., the lowest CR and the shortest time almost for all scenarios.}

\subsubsection{Effect of larger sizes} 
{Finally, we compare five algorithms for solving problems in larger sizes by fixing $(\rho,k_0)=(0.5,10)$ but altering $m\in\{1000,2000,3000,4000\}$ for both examples. In addition, we also fix $n=1000$ for Example \ref{ex-linear}. Now according to the generation of Example \ref{ex-linear}, each instance has $n=1000$ features and $d\in[50000, 600000]$ total samples. We use dataset {\tt rls} for Example \ref{ex-logist} since it has much more features and samples (i.e., $n=20958$ and $d=72309$). Again, we report the results in terms of the median values of 20 instances in Figure \ref{fig:effect-hig}. For Example \ref{ex-linear},  there are five descending trends for CR and ascending trends for the time. However, for Example \ref{ex-logist}, the larger $m$ the higher CR and the time. We find that \ADMM\ runs much faster than the others (e.g., when $m=4000$, \ADMM, \FP, \FS, \FAT, and \FA\ consume 73,	728,	 4711,	3834, and	475 seconds, respectively).
As always, \ADMM\ produces the most desirable results.}

\begin{figure}[h]
	\begin{subfigure}{.245\textwidth}
	\centering
	\includegraphics[width=.995\linewidth]{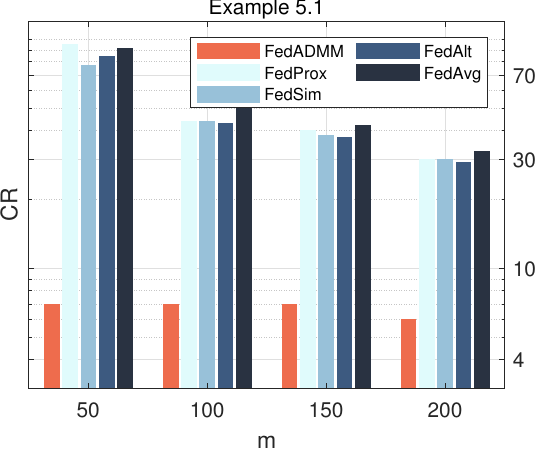}
\end{subfigure}	
\begin{subfigure}{.245\textwidth}
	\centering
	\includegraphics[width=.995\linewidth]{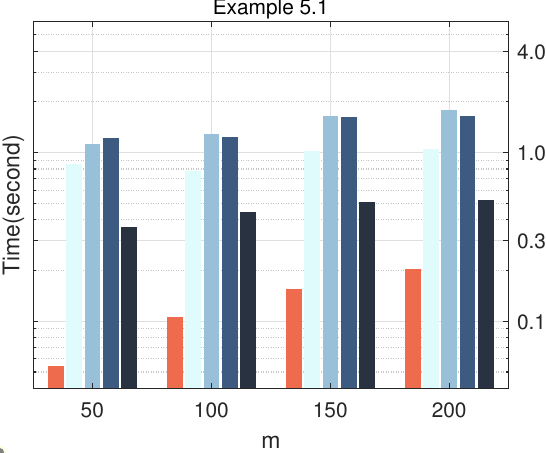}
\end{subfigure} 
~\vspace{1mm}\\ 
	\begin{subfigure}{.245\textwidth}
	\centering
	\includegraphics[width=.995\linewidth]{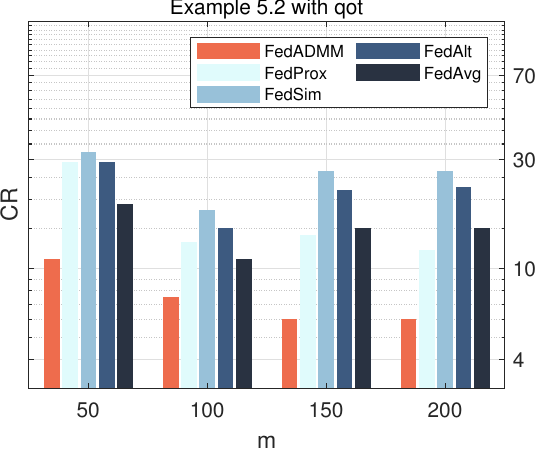}
\end{subfigure}	
\begin{subfigure}{.245\textwidth}
	\centering
	\includegraphics[width=.995\linewidth]{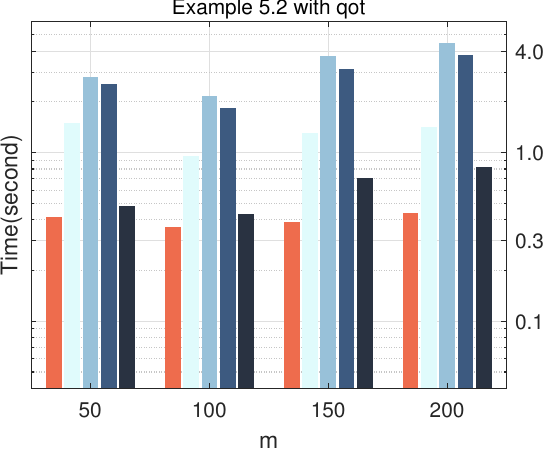}
\end{subfigure} \vspace{-2mm}
\caption{Effect of $m$.\label{fig:effect-m}}
 
\end{figure}

\begin{figure}[!th]
	\begin{subfigure}{.245\textwidth}
	\centering
	\includegraphics[width=.995\linewidth]{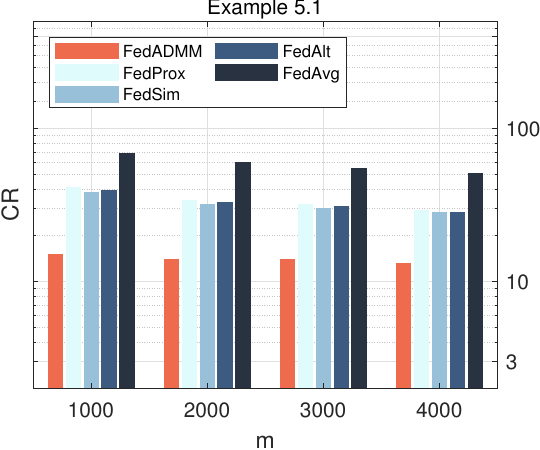}
\end{subfigure}	
\begin{subfigure}{.245\textwidth}
	\centering
	\includegraphics[width=.995\linewidth]{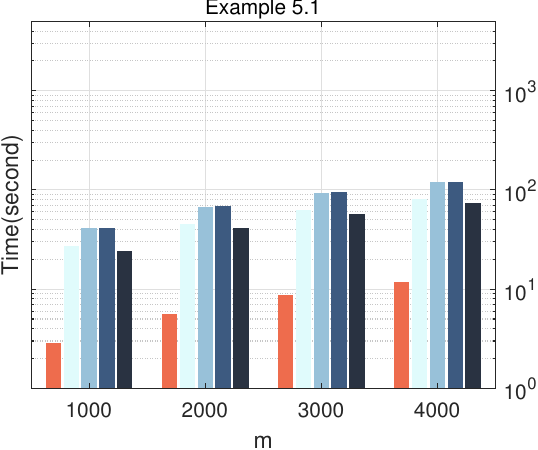}
\end{subfigure} 
~\vspace{1mm}\\ 
	\begin{subfigure}{.245\textwidth}
	\centering
	\includegraphics[width=.995\linewidth]{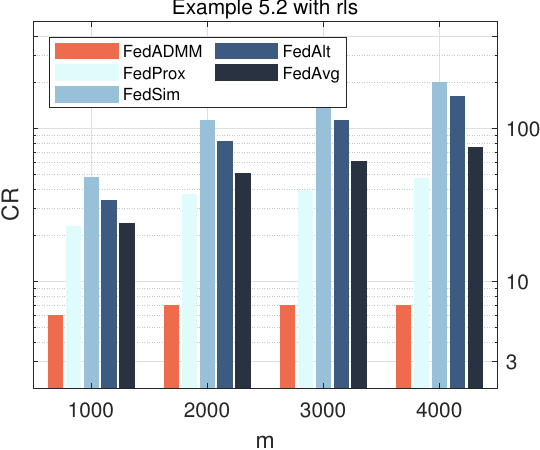}
\end{subfigure}	
\begin{subfigure}{.245\textwidth}
	\centering
	\includegraphics[width=.995\linewidth]{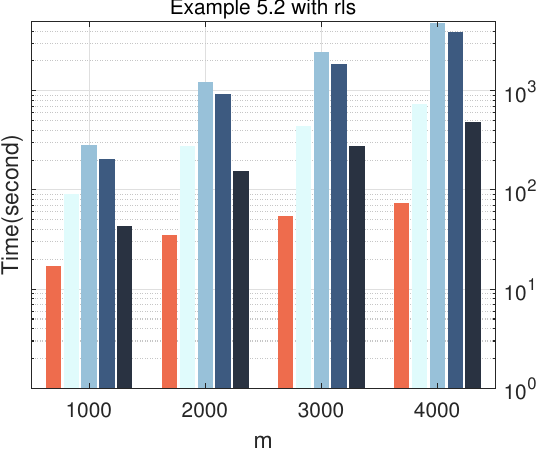}
\end{subfigure} \vspace{-2mm}
\caption{Effect of larger sizes.\label{fig:effect-hig}}\vspace{-4mm}
 \end{figure}
 
%\begin{table}[!th]
%	\renewcommand{\arraystretch}{1.25}\addtolength{\tabcolsep}{0pt}
%	\caption{Effect of higher dimensions.}\vspace{-3mm}
%	\label{tab:high}
%	\begin{center}
%		\begin{tabular}{lllrrrr }
%			\hline
%Data&Datasets&	Source	&	$n$	&	$d$\\\hline
%%\texttt{gis} & Gisette& libsvm& 5000& 6000\\
%\texttt{qot}&	Qsar oral toxicity	&	uci	&	1024 	&	8992 	\\
%% \texttt{sct}&	Santander customer transaction	&	kaggle	&	200 	&	200000 	\\
% {\texttt{higg}}&	 {higgs}	&	 {uci}	&	 {28} 	&	 {$11,000,000$} 	\\
%% \texttt{rtb}&	Real time bidding	&	kaggle	&	88 	&	 {1000000} 	\\
%\hline
% 		\end{tabular}
%	\end{center}
%	\vspace{-5mm}
%\end{table}
\section{Conclusion}
We developed an inexact ADMM-based FL algorithm. The periodic global averaging allows it to reduce CR so as to save communication resources. Solving sub-problems inexactly alleviates clients' computational burdens significantly, thereby accelerating the learning process. Partial device participation in the algorithm eliminates the stragglers' effect. Those merits show the strong potential of \ADMM\ for real-world applications like vehicular communications, mobile edge and over-the-air computing. % Moreover, we feel that the proposed algorithm could be useful to deal with decentralized FL \cite{elgabli2020fgadmm,ye2021decentralized}. We leave these for future research.

\bibliographystyle{IEEEtran}
\bibliography{ref}

%  \title{Supplementary  Material for \\
%`` Federated Learning via Inexact ADMM"
%%Communication and Computational Efficiency, and Global Convergence
%}
%\maketitle

%\newpage
%\pagenumbering{arabic}
%\numberwithin{equation}{section}
\appendices
 \section{Some basics}
% \section{Some Useful Properties}
 For any $\bx_1, \bx_2  $, and $\bx_i  \in\{\bx_1, \bx_2  \}$, it follows that $\bx_2  +t(\bx_1-\bx_2  ) -  \bx_i = (t-1)(\bx_1-\bx_2  )$  or $= t(\bx_1-\bx_2  )$. If function is gradient Lipschitz continuous with constant $r$, then the Mean Value Theorem suffices to
\begin{eqnarray}  \label{H-Lip-continuity-fxy}
 \arraycolsep=1.0pt\def\arraystretch{1.5}  \begin{array}{llll}
&&f(\bx_1) - f(\bx_2    ) -\langle \nabla  f(\bx_i  ), \bx_1-\bx_2     \rangle \\
%&=&  \int_0^1 d f(\bx_2  +t(\bx_1-\bx_2  ))-\langle \nabla  f(\bx_i  ), \bx_1-\bx_2     \rangle\\
&=&  \int_0^1 \langle \nabla  f(\bx_2  +t(\bx_1-\bx_2  )) - \nabla  f(\bx_i ), \bx_1-\bx_2     \rangle dt  \\
&\leq&  \int_0^1 r\|\bx_2  +t(\bx_1-\bx_2  ) -  \bx_i \|\| \bx_1-\bx_2     \| dt\\
%&=&r\| \bx_1-\bx_2\|^2\Bigg\{  
%\begin{array}{lll}
%\int_0^1 (1-t)dt,&~~{\rm if}~&\bx_i=\bx_1\\
%\int_0^1 tdt,&~~{\rm if}~&\bx_i=\bx_2\
%\end{array} \\
&=& \frac{r}{2}\| \bx_1-\bx_2     \|^2.
 \end{array}
\end{eqnarray}  
If $\nabla^2  f\succeq -s I$, then the Mean Value Theorem brings out 
\begin{eqnarray}  \label{Lip-continuity-lowbd}
 \arraycolsep=1.0pt\def\arraystretch{1.5}  \begin{array}{llll}
&&\langle \nabla f(\bx_1) -  \nabla f(\bx_2    ) , \bx_1-\bx_2     \rangle \\
&=&  \int_0^1 \langle \nabla^2  f(\bx_2  +t(\bx_1-\bx_2  ))(\bx_1-\bx_2), \bx_1-\bx_2\rangle \rangle dt  \\
&\geq&   -s \| \bx_1-\bx_2     \|^2.
 \end{array}
\end{eqnarray}   
For any  vectors $\bx_i$,  matrix $H\succeq0$, and $t>0$, we have
\begin{eqnarray} \label{two-vecs}
 \arraycolsep=1.0pt\def\arraystretch{1.5}
 \begin{array}{lcl} 
%\|\bx_1-\bx_3\|^2  &-& \|\bx_2-\bx_3\|^2 \\&=&2\langle\bx_1-\bx_2,\bx_1-\bx_3\rangle-   \|\bx_1-\bx_2\|^2,\\
%2\langle \bx_1,H\bx_2\rangle &\leq &  \|\bx_1\|^2_H+\|\bx_2\|^2_H,\\
2 \langle \bx_1, \bx_2\rangle &\leq &  t\|\bx_1\|^2 + ({1}/{t})\|\bx_2\|^2,\\
\|\bx_1+\bx_2\|^2 &\leq& (1+t)\|\bx_1\|^2+(1+ 1/t)\|\bx_2\|^2,\\
\|\sum_{i=1}^m\bx_i\|^2 &\leq& m\sum_{i=1}^m\|\bx_i\|^2 .
 \end{array}
\end{eqnarray}
For notational simplicity, hereafter, we denote
\begin{eqnarray} 
 \label{decreasing-property-0}   
 \arraycolsep=1.0pt\def\arraystretch{1.5}
 \begin{array}{lllllll}
\triangle\bx^{\tk}&:=&\bx^{\tk}-\bx^{\tau_k}~&\triangle\blx_i^{k+1}&:=&\bx^{k+1}_i-\bx^{\tk},~~\\
\triangle\bx_i^{k+1}&:=&\bx_i^{k+1}-\bx_i^{k},~&\triangle\bpi_i^{k+1}&:=&\bpi_i^{k+1}-\bx_i^{k},\\
\triangle\epsilon_i^{k+1}&:=&\epsilon_i^{k+1}-\epsilon_i^{k},
%\bg_i^{k+1}&:=&\alpha_{i}\nabla f_i(\bx_i^{k+1}),%~&\blg_i^{k+1}&:=&\alpha_{i}\nabla f_i(\bx^{\tk}),
% \bva^{k+1}_i:=\bg_i^{k+1}  +  \bpi_i^{k}+\sigma(\bx_i^{k+1}-\bx^{\tk}),
    \end{array}
 \end{eqnarray} 
 and let   $\bx^{k} \rightarrow \bx$ stand for $\lim_{k\rightarrow\infty} \bx^{k} = \bx$.  In the sequel, for notational simplicity we write
\begin{eqnarray*} 
 \arraycolsep=1.0pt\def\arraystretch{1.5}
 \begin{array}{lllllll}
\sum_i:=\sum_{i=1}^m.
    \end{array}
 \end{eqnarray*} 
 %\section{Proofs of theorems in Section \ref{sec:ADMMFL}}
 \section{Proof of  Theorem \ref{finite-stop}} 
{ The main idea to prove this theorem is to show  \begin{eqnarray}\label{v-ell-v-*}
   \arraycolsep=0pt\def\arraystretch{1.5}
   \begin{array}{l} 
   \|\bv _i^{\ell+1}-\bv_i^{*} \|^2  \leq \varrho^{-1} \|\bv _i^{\ell}-\bv_i^{*} \|^2, \end{array}
\end{eqnarray} which can be verified by using the optimality condition of problem \eqref{update-w-l} that is satisfied by $\bv _i^{\ell+1}$ .}  
 \begin{proof} Since $\bv_i^{\ell+1}$ is a solution to problem \eqref{update-w-l}, it satisfies the following optimality condition, 
 \begin{eqnarray*}%\label{update-w-l-opt}
   \arraycolsep=0pt\def\arraystretch{1.5}
 \begin{array}{llll}
\alpha_{i}   \nabla f_i(\bv _i^{\ell}) + \bpi_i^k + \sigma_i(\bv_i^{\ell+1}-\bx^{\tk})+  \alpha_{i}r_i(\bv_i^{\ell+1}-\bv _i^{\ell})=0,
 \end{array}
\end{eqnarray*}
which subtracting \eqref{update-w-*} gives rise to
 \begin{eqnarray*}
   \arraycolsep=1.0pt\def\arraystretch{1.5}
 \begin{array}{llll}
&-&\sigma_i(\bv_i^{\ell+1}-\bv_i^{*})- \alpha_{i}r_i(\bv_i^{\ell+1}-\bv _i^{\ell})\\
&=&\alpha_{i} (\nabla f_i(\bv _i^{\ell})- \nabla f_i(\bv_i^{*}))\\
&=&\alpha_{i} (\nabla f_i(\bv _i^{\ell})- \nabla f_i(\bv_i^{\ell+1})+\nabla f_i(\bv _i^{\ell+1})- \nabla f_i(\bv_i^{*})).
 \end{array}
\end{eqnarray*}
Using the condition allows us to obtain
 \begin{eqnarray*}
   \arraycolsep=1.0pt\def\arraystretch{1.5}
 \begin{array}{lcl}
&-&\alpha_{i}s_i\|\bv _i^{\ell+1}-\bv_i^{*} \|^2\\
&\overset{\eqref{Lip-continuity-lowbd}}{\leq}&
\langle  \bv _i^{\ell+1}-\bv_i^{*} , \alpha_{i} (\nabla f_i(\bv _i^{\ell+1})- \nabla f_i(\bv_i^{*}) \rangle\\ 
&=&
\langle \bv _i^{\ell+1}-\bv_i^{*}, -\sigma_i(\bv_i^{\ell+1}-\bv_i^{*})- \alpha_{i}r_i(\bv_i^{\ell+1}-\bv _i^{\ell}) \rangle\\ 
&+&\langle \bv _i^{\ell+1}-\bv_i^{*} , \alpha_{i} (\nabla f_i(\bv _i^{\ell+1})- \nabla f_i(\bv_i^{\ell})  \rangle\\ 
&=&
\langle \bv _i^{\ell+1}-\bv_i^{*}, -\sigma_i(\bv_i^{\ell+1}-\bv_i^{*})- \alpha_{i}r_i(\bv_i^{\ell+1}-\bv _i^{\ell}) \rangle\\ 
&+&\alpha_{i} \langle \sqrt{r_i}(\bv _i^{\ell+1}-\bv_i^{*} ), \sqrt{1/r_i}  (\nabla f_i(\bv _i^{\ell+1})- \nabla f_i(\bv_i^{\ell})  \rangle\\ 
&\leq& -(\sigma_i+\alpha_{i}r_i) \|\bv _i^{\ell+1}-\bv_i^{*} \|^2- \alpha_{i}r_i\langle \bv _i^{\ell+1}-\bv_i^{*},  \bv_i^{*}-\bv _i^{\ell} \rangle\\
&+&\frac{\alpha_{i}r_i}{2}(\|\bv _i^{\ell+1}-\bv_i^{*} \|^2+\|\bv _i^{\ell+1}-\bv_i^{\ell} \|^2)\\
&=& -(\sigma_i+\alpha_{i}r_i/2) \|\bv _i^{\ell+1}-\bv_i^{*} \|^2\\
&+&\frac{\alpha_{i}r_i}{2}( \|\bv _i^{\ell+1}-\bv_i^{*}+  \bv_i^{*}-\bv_i^{\ell} \|^2- 2\langle \bv _i^{\ell+1}-\bv_i^{*},  \bv_i^{*}-\bv _i^{\ell} \rangle)\\
&=& -\sigma_i\|\bv _i^{\ell+1}-\bv_i^{*} \|^2 + \frac{\alpha_{i}r_i}{2} \|\bv _i^{\ell}-\bv_i^{*} \|^2\\
&\leq& -(\alpha_{i} s_i+\frac{\varrho \alpha_{i}r_i}{2}) \|\bv _i^{\ell+1}-\bv_i^{*} \|^2 + \frac{\alpha_{i}r_i}{2} \|\bv _i^{\ell}-\bv_i^{*} \|^2,
 \end{array}
\end{eqnarray*}
which immediately results in \eqref{v-ell-v-*}, thereby leading to
 \begin{eqnarray*}
   \arraycolsep=1.0pt\def\arraystretch{1.5}
 \begin{array}{lcccc}
\|\bv _i^{\kappa+1}-\bv_i^{*} \|^2&\leq&\frac{1}{\varrho}\|\bv _i^{\kappa}-\bv_i^{*} \|^2 &\leq& \frac{1}{\varrho^2}\|\bv _i^{\kappa-1}-\bv_i^{*} \|^2\\
&\leq& \cdots&\leq&\frac{1}{\varrho^{\kappa+1}}\|\bv _i^{0}-\bv_i^{*} \|^2.
 \end{array}
\end{eqnarray*}
Now letting $ \bx^{k+1}_i= \bv^{\kappa+1}_i$,  we verify the condition in \eqref{iceadmm-sub2} by
 \begin{eqnarray*}
   \arraycolsep=1.0pt\def\arraystretch{1.5}
 \begin{array}{lcl}
&& \|\bg_i^{k+1}  +  \bpi_i^{k}+\sigma_i(\bx_i^{k+1}-\bx^{\tk})\|^2\\
&\overset{\eqref{update-w-*}}{=}&\|\bg_i^{k+1} - \alpha_{i}\nabla f_i(\bv_i^{*})  +  \sigma_i(\bx_i^{k+1}-\bv_i^{*})\| ^2\\
&\leq&2(\alpha_{i}^2r_i^2+\sigma_i^2)\|\bx_i^{k+1}-\bv_i^{*}\| ^2\\
&\leq&2(\alpha_{i}^2r_i^2+\sigma_i^2)/{\varrho^{\kappa+1}}\|\bv _i^{0}-\bv_i^{*} \|^2
 \leq \epsilon_i^{k+1},
 \end{array}
\end{eqnarray*}
which by $\bv _i^{0}=\bx^{\tk}$ implies that
 \begin{eqnarray*}
   \arraycolsep=1.0pt\def\arraystretch{1.5}
 \begin{array}{lcl}
\kappa = \log_{\varrho}\left \lceil \frac{2(\alpha_{i}^2r_i^2+\sigma_i^2)\|\bx^{\tk} -\bv_i^{*}\|^2}{\epsilon_i^{k+1}}  \right \rceil -1.
 \end{array}
\end{eqnarray*}  
The whole proof is finished.
\end{proof}   

\section{Proofs of theorems in Section \ref{sec:convergence}}

\subsection{Key lemma}
\begin{lemma}\label{basic-observations}  The following statements are valid.

\noindent a) For any  $k\in\K$,
\begin{eqnarray}
 \label{opt-con-xk1-1}
\begin{array}{rcll}
\sum_{i}  ( \sigma_i(\bx_i^{k}-\bx^{\tau_{k+1}})+\bpi_i^{k})=0.
\end{array} \end{eqnarray} 
b) For  any  $k\geq 0$ and any $i\in[m]$,
\begin{eqnarray}
 \label{opt-con-xk1-2}
\begin{array}{rcll} 
{\|\bva_i^{k+1}\|^2\leq \epsilon_i^{k+1},~~\text{where}~~ \bva_i^{k+1}:=\bg_i^{k+1}  +  \bpi_i^{k+1}.}
\end{array} \end{eqnarray} 
c) Under Assumption \ref{ass-fi}, for  any  $k\geq 0$ and any $i\in[m]$,
\begin{eqnarray}
 \label{opt-con-xk1-3}
 \arraycolsep=0pt\def\arraystretch{1.5}
\begin{array}{rcll} 
 \|\triangle \bpi_i^{k+1} \|^2  \leq   \frac{6\alpha_{i}^2r_i^2}{5} \|\triangle\bx_i^{k+1} \|^2 -\frac{24}{1-\nu_{i}}\triangle\epsilon_i^{k+1}.
\end{array} \end{eqnarray} 
d) Under Scheme \ref{ass-omega}, for any $i\in[m]$,  
\begin{eqnarray}
 \label{opt-con-xk1-4}
\begin{array}{rcll}
\epsilon_i^{k+1}\to 0.
\end{array} \end{eqnarray} 
\end{lemma}  
\begin{proof} a) For any $i\in[m]$ and at $(k+1)$th iteration, let $k_i$ be the largest integer in $[-1,k]$ such that $i\in\Omega^{\tau_{k_i+1}}$. This implies that client $i$ is not selected in all $\Omega^{\tau_{k_i+2}}, \Omega^{\tau_{k_i+3}} \cdots ,\Omega^{\tk}$, which by \eqref{iceadmm-sub5}  yields
 \begin{eqnarray}\label{par-invariance}
\arraycolsep=0pt\def\arraystretch{1.5}
\begin{array}{rcl}
(\epsilon_i^{\ell+1}, \bx^{\ell+1}_i, \bpi^{\ell+1}_i, \bz^{\ell+1}_i)&\equiv&(\epsilon_i^{k_i+1}, \bx^{k_i+1}_i,  \bpi^{k_i+1}_i, \bz^{k_i+1}_i),\\
\forall \ell&= &k_i,k_i+1, \cdots ,k.
\end{array} \end{eqnarray}
For any client $i\in\tauk$, we have \eqref{iceadmm-sub4}. For any client $i\notin\Omega^{\tk}$, if ${k_i   }\geq0$,  then $( \bx^{k_i+1}_i,  \bpi^{k_i+1}_i, \bz^{k_i+1}_i)$ also satisfies \eqref{iceadmm-sub4} due to $i\in\Omega^{\tau_{k_i+1}}$, which by condition \eqref{par-invariance} implies that $(\bx^{k+1}_i, \bpi^{k+1}_i, \bz^{k+1}_i)$ satisfies \eqref{iceadmm-sub4}.  If $k_i=-1$, this means that is client $i$ has never been selected. Then by \eqref{par-invariance} and our initialization, we have  \begin{eqnarray*} 
\arraycolsep=1.0pt\def\arraystretch{1.5}
\begin{array}{lcl}\bz^{k+1}_i=\bz^{k_i+1}_i=\bz^{0}_i =\sigma_i\bx_i^{0} +\bpi_i^{0}=\sigma_i\bx_i^{k+1} +\bpi_i^{k+1}.\end{array} \end{eqnarray*} Hence, \eqref{iceadmm-sub4} is still valid. Overall, we can conclude that \eqref{iceadmm-sub4}  holds for every $i\in[m]$ and $k\geq 0$. Now, for any $k\in\K$,  
 \begin{eqnarray*}
\arraycolsep=1.0pt\def\arraystretch{1.5}
\begin{array}{lcl}
\sum_{i}  (  \sigma_i(\bx_i^{k}-\bx^{\tau_{k+1}})+\bpi_i^{k} )
 \overset{\eqref{iceadmm-sub4}}{=} \sum_{i} \sigma_i( \bz_i^{k} -\bx^{\tau_{k+1}})   \overset{\eqref{iceadmm-sub1}}{=} 0.
\end{array} \end{eqnarray*} 
b) For any  $i\in\tauk$,  solution $\bx_i^{k+1}$ in \eqref{iceadmm-sub2}  satisfies  
 \begin{eqnarray}
 \label{opt-con-xik1-10}
 \arraycolsep=1.0pt\def\arraystretch{1.5}
\begin{array}{rcl}
\bva_i^{k+1} = \bg_i^{k+1}  +  \bpi_i^{k+1} 
&\overset{\eqref{iceadmm-sub3}}{=}& \bg_i^{k+1} + \bpi_i^{k}  +\sigma_i \triangle\blx_i^{k+1},\\
\|\bva_i^{k+1} \|^2&\overset{\eqref{iceadmm-sub20}}{\leq}&\epsilon_i^{k+1}.
   \end{array}
  \end{eqnarray} 
For any $i\notin\tauk$, we have  \begin{eqnarray*} 
 \arraycolsep=1.0pt\def\arraystretch{1.5}
\begin{array}{lcl}
\bva_i^{k_i+1}=\bg_i^{k_i+1}  +  \bpi_i^{k_i+1},~~\|\bva_i^{k_i+1} \|^2\leq\epsilon_i^{k_i+1}    \end{array}
  \end{eqnarray*}  due to  $i\in\Omega^{\tau_{k_i+1}}$. This together with \eqref{par-invariance} implies that \eqref{opt-con-xk1-2} is still true.  So, \eqref{opt-con-xk1-2} holds for any $i\in[m]$ and any $k\geq 0$.

c) For any $i\in\tauk$, it follows from \eqref{opt-con-xk1-2} and the gradient Lipschitz continuity of $f_i$ that
 \begin{eqnarray} \label{gap-pi-k-eps}
 \arraycolsep=1.0pt\def\arraystretch{1.5}
\begin{array}{lcl}
\|\triangle   \bpi_i^{k+1}\|^2 &\leq& \| \bg_i^{k+1} - \bg_i^{k} +\bva_i^{k+1}-\bva_i^{k}\|^2\\
&\overset{\eqref{two-vecs}}{\leq}& \frac{6\alpha_{i}^2r_i^2}{5}  \|\triangle\bx_i^{k+1} \|^2+6\| \bva_i^{k+1}-\bva_i^{k}\|^2\\
&\overset{\eqref{iceadmm-sub20}}{\leq}& \frac{6\alpha_{i}^2r_i^2}{5}  \|\triangle\bx_i^{k+1} \|^2+12(\epsilon_i^{k+1}+\epsilon_i^{k})\\
&\overset{\eqref{iceadmm-sub20}}{\leq}&\frac{6\alpha_{i}^2r_i^2}{5} \|\triangle\bx_i^{k+1} \|^2+\frac{12(1+\nu_{i})}{1-\nu_{i}}(\epsilon_i^{k}-\epsilon_i^{k+1})\\&\overset{\eqref{iceadmm-sub20}}{\leq}&\frac{6\alpha_{i}^2r_i^2}{5} \|\triangle\bx_i^{k+1} \|^2-\frac{24}{1-\nu_{i}}\triangle\epsilon_i^{k+1}.
   \end{array}
  \end{eqnarray}
  For any $i\notin\tauk$, we have  $\triangle\bpi_i^{k+1}=0$ by \eqref{iceadmm-sub5}, and thus the above condition is still valid.
  
  d) For sufficiently large $k\in\K$, any client $i$ has been selected at least $k/(s_0k_0)$ times due to \eqref{scheme-omega-2T}. This means that \eqref{iceadmm-sub20} (e.g., $\epsilon_i^{k+1}=\epsilon_i^{k}/2$ ) occurs at least $k/(s_0k_0)$ times,  thereby leading to $\epsilon_i^{k+1}\to 0$. 
\end{proof}
\subsection{Proof of Lemma \ref{lemma-decreasing-1}}
{
To estimate the upper bound of gap $(\L^{k+1}-\L^{k})$, we first decompose it into three pieces as follows 
 \begin{eqnarray} 
\label{three-cases}  
 \qquad   \begin{array}{llllll}
 \L^{k+1}-\L^{k}  =: p_1^k+p_2^k+p_3^k,
    \end{array} 
 \end{eqnarray} 
where
   \begin{eqnarray}  \label{three-cases-sub}  
 \arraycolsep=1.0pt\def\arraystretch{1.5} \begin{array}{llllll}
p_1^k&:=&\L(\bx^{\tk},W^{k},\Pi^{k})-\L(\bx^{\tau_k},W^{k},\Pi^{k}),\\
p_2^k&:=& \L(\bx^{\tk},W^{k+1},\Pi^{k})-\L(\bx^{\tk},W^{k},\Pi^{k}), \\
p_3^k&:=& \L(\bx^{\tk},W^{k+1},\Pi^{k+1})-\L(\bx^{\tk},W^{k+1},\Pi^{k}).
    \end{array} 
 \end{eqnarray} 
 Then we apply condition \eqref{opt-con-xk1-1} for $\bx^{\tk}$,  condition \eqref{opt-con-xk1-2} for $W^{k+1}$, and  conditions \eqref{iceadmm-sub3} and \eqref{iceadmm-sub5}  for $\Pi^{k+1}$ to derive the upper bounds for $p_1^k$, $p_2^k$, and $p_3^k$, respectively. Finally, adding these bounds yields the upper bound of gap $(\L^{k+1}-\L^{k})$. The detailed proof is given as follows. }
\begin{proof}  
%We estimate gap $(\L^{k+1}-\L^{k})$  by decomposing it as 
%  \begin{eqnarray} 
%\label{three-cases}  
% %\arraycolsep=1.0pt\def\arraystretch{1.5}
% \qquad   \begin{array}{llllll}
% \L^{k+1}-\L^{k}  =: p_1^k+p_2^k+p_3^k,
%    \end{array} 
% \end{eqnarray} 
% with
%   \begin{eqnarray}  \label{three-cases-sub}  
% \arraycolsep=1.0pt\def\arraystretch{1.5} \begin{array}{llllll}
%p_1^k&:=&\L(\bx^{\tk},W^{k},\Pi^{k})-\L^{k},\\
%p_2^k&:=& \L(\bx^{\tk},W^{k+1},\Pi^{k})-\L(\bx^{\tk},W^{k},\Pi^{k}), \\
%p_3^k&:=& \L^{k+1}-\L(\bx^{\tk},W^{k+1},\Pi^{k}).
%    \end{array} 
% \end{eqnarray} 
 Since $\sigma_i\geq 3\alpha_i r_i > 0$, it follows
 \begin{eqnarray} 
\label{sigma-a-r} \begin{array}{lll}  
  0< \alpha_i r_i\leq  \frac{\sigma_i}{3}, ~~\forall i\in[m].
   \end{array}
 \end{eqnarray}
{\underline{Estimate  $p_1^k$.}} 
If  $k\in\K$, then we have \eqref{opt-con-xk1-1}  and
 \begin{eqnarray} \label{tv-0}
   \arraycolsep=1.0pt\def\arraystretch{1.5}
  \begin{array}{lcl}
 && \frac{\sigma _i}{2}\|\bx_i^k-\bx^{\tk}\|^2 -\frac{\sigma _i}{2}\|\triangle\blx_i^{k}\|^2\\
  &=&\frac{\sigma _i}{2}\|\bx_i^k-\bx^{\tk}\|^2 -\frac{\sigma _i}{2}\|\bx_i^k-\bx^{\tk}+ \triangle\bx^{\tk}\|^2\\
  &=& \langle\triangle\bx^{\tk},    -\sigma _i(\bx_i^k-\bx^{\tk})\rangle-  \frac{\sigma_i}{2} \| \triangle\bx^{\tk}\|^2.
    \end{array} 
 \end{eqnarray} 
The  fact allows  us to derive that
\begin{eqnarray}  \label{two-cases-3}  
 \arraycolsep=1.0pt\def\arraystretch{1.5}
 \begin{array}{lcl}
  p_1^k  &\overset{\eqref{Def-L}}{=}& \sum_{i}    \langle  \triangle\bx^{\tk}, -\bpi_i^k\rangle\\
    &+&  \sum_{i} (  \frac{\sigma _i}{2}\|\bx_i^k-\bx^{\tk}\|^2 -\frac{\sigma _i}{2}\|\bx_i^k-\bx^{\tau_k}\|^2 )\\
    &\overset{\eqref{tv-0}}{=}& \sum_{i}    \langle\triangle\bx^{\tk}, -\bpi_i^k-\sigma _i(\bx_i^k-\bx^{\tk})\rangle\\
&-& \sum_{i} \frac{\sigma_i}{2} \| \triangle\bx^{\tk}\|^2 \\
&\overset{\eqref{opt-con-xk1-1}}{=}&- \sum_{i} \frac{\sigma_i}{2} \| \triangle\bx^{\tk}\|^2 \\
&\leq& - \sum_{i} \frac{\sigma_i}{10} \| \triangle\bx^{\tk}\|^2.
    \end{array} 
 \end{eqnarray}
 If  $k\notin\K$, then $\bx^{\tk}=\bx^{\tau_k}$, thereby aslo leading to
\begin{eqnarray} 
\label{two-cases-1} \begin{array}{lll}  
   p_1^k \overset{\eqref{three-cases-sub}}{=}  0 =  - \sum_{i} \frac{\sigma_i}{10} \| \triangle\bx^{\tk}\|^2. 
   \end{array}
 \end{eqnarray}
{\underline{Estimate  $p_2^k$.}}  We consider two cases: $i\in\tauk$ and $i\notin\tauk$. For case $i\in\tauk$, it follows from \eqref{two-vecs} that
 \begin{eqnarray} \label{tv-1}
   \arraycolsep=1.0pt\def\arraystretch{1.5}
  \begin{array}{lcl}
 && \frac{\sigma_i }{2}\|\triangle\blx_i^{k+1} \|^2 -\frac{\sigma_i }{2}\|\bx_i^{k}-\bx^{\tau_{k+1}}\|^2\\
  &=& \langle \triangle\bx_i^{k+1},   \sigma_i \triangle\blx_i^{k+1}\rangle-  \frac{\sigma_i}{2} \| \triangle\bx_i^{k+1}\|^2.
    \end{array} 
 \end{eqnarray} 
  Let $u_i$ be defined as  follows,
 \begin{eqnarray*}
 \arraycolsep=1.0pt\def\arraystretch{1.5}\begin{array}{lcl}
  u_{i}^k  
   &\overset{\eqref{Def-L}}{:=}&     L(\bx^{\tk},\bw_i^{k+1},\bpi_i^{k})-L(\bx^{\tk},\bw_i^{k},\bpi_i^{k}) \\
   &\overset{\eqref{Def-L}}{=}&      \alpha_{i}  f_i(\bx_i^{k+1})- \alpha_{i} f_i(\bx_i^k)   +  \langle \triangle \bx_i^{k+1} , \bpi_i^k\rangle\\
     &+ &       \frac{\sigma _i}{2}\|\triangle\blx_i^{k+1}\|^2 -\frac{\sigma _i}{2}\|\bx_i^{k}-\bx^{\tk}\|^2 \\ 
      &\overset{\eqref{tv-1},\eqref{iceadmm-sub3} }{=}&      \alpha_{i}  f_i(\bx_i^{k+1})- \alpha_{i} f_i(\bx_i^k)     \\
     &+ &  \langle \triangle \bx_i^{k+1} , \bpi_i^{k+1}\rangle   -  \frac{\sigma_i}{2} \| \triangle\bx_i^{k+1}\|^2.
    \end{array} 
 \end{eqnarray*}
Now, by using the above condition we have  
\begin{eqnarray*}
 \arraycolsep=1.0pt\def\arraystretch{1.5}\begin{array}{lcl}
  u_{i}^k  
%   &\overset{\eqref{Def-L}}{=}&      \alpha_{i}  f_i(\bx_i^{k+1})- \alpha_{i} f_i(\bx_i^k)   +  \langle \triangle \bx_i^{k+1} , \bpi_i^k\rangle\\
%     &+ &       \frac{\sigma _i}{2}\|\triangle\blx_i^{k+1}\|^2 -\frac{\sigma _i}{2}\|\bx_i^{k}-\bx^{\tk}\|^2 \\ 
%      &\overset{\eqref{tv-1},\eqref{iceadmm-sub3} }{=}&      \alpha_{i}  f_i(\bx_i^{k+1})- \alpha_{i} f_i(\bx_i^k)     \\
%     &+ &  \langle \triangle \bx_i^{k+1} , \bpi_i^{k+1}\rangle   -  \frac{\sigma_i}{2} \| \triangle\bx_i^{k+1}\|^2\\  
 &\overset{\eqref{H-Lip-continuity-fxy}}{\leq}& \frac{\alpha_{i}r_i-\sigma_i}{2}\| \triangle\bx_i^{k+1}\|^2 +   \langle \triangle\bx_i^{k+1}, \bg_i^{k+1} + \bpi_i^{k+1}\rangle   \\
 &\overset{\eqref{sigma-a-r} }{\leq}&- \frac{\sigma_i}{3}\| \triangle\bx_i^{k+1}\|^2 +   \langle \triangle\bx_i^{k+1}, \bg_i^{k+1} + \bpi_i^{k+1}\rangle    \\ 
  &\overset{ \eqref{opt-con-xk1-2} }{=}&- \frac{\sigma_i}{3}\| \triangle\bx_i^{k+1}\|^2 + \langle \triangle\bx_i^{k+1}, \bva_i^{k+1}\rangle     \\ 
   &\overset{\eqref{two-vecs}}{\leq}& -\frac{\sigma _i}{3}\| \triangle\bx_i^{k+1}\|^2+ \frac{5}{\sigma_i}\|\bva_i^{k+1}\|^2 +\frac{\sigma _i}{10}\| \triangle\bx_i^{k+1}\|^2  \\
  &\overset{\eqref{opt-con-xk1-2}}{\leq}&-\frac{7\sigma _i}{30}\| \triangle\bx_i^{k+1}\|^2+  \frac{5\epsilon_i^{k+1}}{\sigma_i} \\
    &\overset{\eqref{iceadmm-sub20}}{\leq}& -\frac{7\sigma _i}{30}\| \triangle\bx_i^{k+1}\|^2 - \frac{5\triangle\epsilon_i^{k+1}}{(1-\nu_{i})\sigma_i}.
    \end{array} 
 \end{eqnarray*}
 For case $i\notin\tauk$,   $ \bpi_i^{k+1}=\bpi_i^{k}$ and $\epsilon_i^{k+1}=\epsilon_i^{k}$ from \eqref{iceadmm-sub5} indicate the above condition is also valid. Therefore, for both cases, we obtain
 \begin{eqnarray}\label{gap-2-1}   
 \arraycolsep=0.5pt\def\arraystretch{1.5}\begin{array}{lcl}
  p_{2}^k = \sum_{i}  u_{i}^k  \leq \sum_{i} (-\frac{7\sigma _i}{30}\| \triangle\bx_i^{k+1}\|^2 - \frac{5\triangle\epsilon_i^{k+1}}{(1-\nu_{i})\sigma_i}) .
    \end{array} 
 \end{eqnarray}
% which  by \eqref{Def-L} delivers an upper bound for $e_2^k$ as
% \begin{eqnarray} 
%\label{gap-2-1}  
% \arraycolsep=1.0pt\def\arraystretch{1.5} \begin{array}{lll}
%   e_2^k&=&  \sum_{i}     (  L(\bx^{\tk},\bx_i^{k+1},\bpi_i^{k})  - L(\bx^{\tk},\bx_i^{k},\bpi_i^{k})) \\
%    &\leq& \sum_{i}     ( -\alpha_{i} \| \triangle\bx_i^{k+1}\|^2_{H_i^{k+1}}     - \frac{\sigma _i - \alpha_{i}r_i}{2}\| \triangle\bx_i^{k+1}\|^2). 
%    \end{array} 
% \end{eqnarray}
{\underline{Estimate  $p_3^k$.}} We consider two cases: $i\in\tauk$ and $i\notin\tauk$. For case $i\in\tauk$, it is easy to see that
 \begin{eqnarray*} 
 \arraycolsep=1.0pt\def\arraystretch{1.5}  \begin{array}{lcl}
   v_i^k &\overset{\eqref{Def-L}}{=}&     \langle \triangle \blx_i^{k+1} ,\triangle \bpi_i^{k+1}  \rangle
   \overset{\eqref{iceadmm-sub3}}{=}    \frac{1}{\sigma _i}\|\triangle \bpi_i^{k+1} \|^2 \\
   &
   \overset{\eqref{opt-con-xk1-3}}{\leq}&   \frac{6\alpha_{i}^2r_i^2}{5\sigma _i}  \|\triangle\bx_i^{k+1} \|^2- \frac{24}{(1-\nu_{i})\sigma_i} \triangle\epsilon_i^{k+1}\\
    &
   \overset{\eqref{sigma-a-r}}{\leq}&   \frac{4\sigma_i}{30}  \|\triangle\bx_i^{k+1} \|^2- \frac{24}{(1-\nu_{i})\sigma_i} \triangle\epsilon_i^{k+1}.
    \end{array} 
 \end{eqnarray*}
For case $i\notin\tauk$,  the last inequality in the above condition is also valid due to $ \bpi_i^{k+1}{=}\bpi_i^{k}$ and $\epsilon_i^{k+1}{=}\epsilon_i^{k}$ from \eqref{iceadmm-sub5}.  Therefore, for both cases we obtain
  \begin{eqnarray} 
\label{gap-3-1}  
 \arraycolsep=0.5pt\def\arraystretch{1.5}  \begin{array}{lcl}
   p_3^k =  \sum_{i} v_i^k   \leq    \sum_{i}  ( \frac{4\sigma_i}{30}  \|\triangle\bx_i^{k+1} \|^2- \frac{24\triangle\epsilon_i^{k+1}}{(1-\nu_{i})\sigma_i} ).
    \end{array} 
 \end{eqnarray}
 Overall, combining conditions \eqref{three-cases},  \eqref{two-cases-3}, \eqref{two-cases-1},  and \eqref{gap-2-1},  \eqref{gap-3-1},  
  we obtain
  \begin{eqnarray*}  
   \arraycolsep=1.0pt\def\arraystretch{1.5}
  \begin{array}{lll}
   && \L^{k+1}-\L^{k}= p_1^k+p_2^k+p_3^k\\
    &\leq&  - \sum_{i} \left(\frac{\sigma _i}{10} (\|\triangle \bx^{\tk} \|^2+   \|\triangle\bx_i^{k+1}\|^2)-\frac{29\triangle\epsilon_i^{k+1}}{(1-\nu_{i})\sigma_i}\right),
    \end{array} 
 \end{eqnarray*}
 showing  the desired  result.
\end{proof}
  
{Lemma \ref{lemma-decreasing-1} allows us to directly conclude the non-increasing property of sequence $\{ \widetilde\L^k\}$ and vanishing of several gaps  $\triangle \bx^{\tk},\triangle\blx_i^{k+1} ,\triangle \bx^{k+1}_i$, and $\triangle \bpi_i^{k+1}$, as shown by the following results.
}
\begin{lemma}\label{L-bounded-decreasing-1} Suppose that Assumptions \ref{ass-fi} and \ref{ass-f} hold. Every client  $i\in[m]$ chooses $\sigma_i\geq 3\alpha_{i}r_i$ and the sever selects  $\Omega^{\tau_k}$ as Scheme \ref{ass-omega}. Then the following results hold.
 \begin{itemize}
 \item[a)] Sequence $\{ \widetilde\L^k\}$ is non-increasing. 
 \item[b)] $ \widetilde\L^k\geq f(\bx^{\tau_k})\geq f^* >-\infty$ for any $ k\geq1$.
 \item[c)]The limits of all the following terms are zero, namely,  
\begin{eqnarray} \label{limit-5-term-0-inexact}
 \arraycolsep=1.0pt\def\arraystretch{1.5} 
 \begin{array}{lll} (\epsilon_i^{k+1}, \triangle \bx^{\tk},~\triangle\blx_i^{k+1} ,~ \triangle \bx^{k+1}_i,~ \triangle \bpi_i^{k+1})\to 0.\
    \end{array} 
 \end{eqnarray}
% \item[iv)] ${\lim}_{k \rightarrow\infty}( \bx_i^{k+1}-\bx_i^{\tau_k+1}) =0$.
 \end{itemize}
\end{lemma}  
\begin{proof}a) The conclusion follows Lemma \ref{lemma-decreasing-1} immediately. 

b) It follows from \eqref{two-vecs} where $t=\frac{2\sigma_{i}}{3}$ that
\begin{eqnarray} \label{two-vecsss}
 \arraycolsep=1pt\def\arraystretch{1.5}
  \begin{array}{ccll}
 \langle \triangle \blx_i^{k},   \bva_i^k  \rangle \leq  \frac{\sigma_{i}}{3}\|\triangle \blx_i^{k}\|^2+\frac{3}{2\sigma_{i}}\|\triangle \bva_i^k \|^2.
    \end{array} 
 \end{eqnarray} 
 The gradient Lipschitz continuity of $f_i$   implies
\begin{eqnarray*} 
 \arraycolsep=1pt\def\arraystretch{1.5}
  \begin{array}{ccll}
 && \alpha_{i}f_i(\bx^{\tau_k})- \alpha_{i}f_i(\bx_i^{k})\\
  &\overset{\eqref{H-Lip-continuity-fxy}}{\leq}& \langle \triangle \blx_i^{k},  \bg_i^{k} \  \rangle    + \frac{r_i\alpha_{i}}{2}\|\triangle \blx_i^{k}\|^2 \\
  &\overset{\eqref{sigma-a-r}}{\leq}& \langle \triangle \blx_i^{k},  \bg_i^{k} \  \rangle    + \frac{\sigma_{i}}{6}\|\triangle \blx_i^{k}\|^2 \\
 &\overset{\eqref{opt-con-xk1-2}}{=}&  \langle \triangle \blx_i^{k},  \bpi_i^{k} +\bva_i^k  \rangle    + \frac{\sigma_{i}}{6} \|\triangle \blx_i^{k}\|^2\\
  &\overset{\eqref{two-vecsss}}{\leq}&  \langle \triangle \blx_i^{k},  \bpi_i^{k}  \rangle    + \frac{\sigma_{i}}{3}\|\triangle \blx_i^{k}\|^2+\frac{3}{2\sigma_i}\|\bva_i^{k}\|^2 + \frac{\sigma_{i}}{6} \|\triangle \blx_i^{k}\|^2 \\
  &\overset{\eqref{opt-con-xk1-2}}{\leq}&  \langle \triangle \blx_i^{k},  \bpi_i^{k}  \rangle    +\frac{\sigma_{i}}{2}\|\triangle \blx_i^{k}\|^2+\frac{3\epsilon_i^k}{2\sigma_i}.
    \end{array} 
 \end{eqnarray*} 
The above relation and $1>\nu_{i}\geq1/2$ give rise to
\begin{eqnarray*} 
 \arraycolsep=1.0pt\def\arraystretch{1.5}
  \begin{array}{ccll}
 \widetilde \L^{k} 
  &\overset{\eqref{def-c}}{=}&  \sum_{i}   (L(\bx^{\tau_k},\bx_i^{k},\bpi_i^{k})+\frac{29\epsilon_i^k}{(1-\nu_{i})\sigma_i})\\
 &\overset{\eqref{Def-L}}{\geq}&  \sum_{i} (\alpha_{i}  f_i(\bx_i^{k}) + \langle \triangle \blx_i^{k}, \bpi_i^{k}\rangle + \frac{\sigma _i}{2}\|\triangle \blx_i^{k}\|^2+\frac{3\epsilon_i^k}{2\sigma_i})\\ 
% &\geq& \sum_{i} (\alpha_{i}f_i(\bx^{\tau_k})+\frac{2\sigma_i-3r_i\alpha_{i}}{4}\|\triangle \blx_i^{k}\|^2)\\
  &\geq& \sum_{i}  \alpha_{i}f_i(\bx^{\tau_k}) = f(\bx^{\tau_k})  \geq f^*   \overset{\eqref{FL-opt-lower-bound}}{ >} -\infty.
    \end{array} 
 \end{eqnarray*}
c)  Using  Lemma \ref{lemma-decreasing-1}  and $\widetilde \L^{k}>-\infty$ enables to show  that
 \begin{eqnarray*} %\label{bounded-L-f}
 \arraycolsep=1.0pt\def\arraystretch{1.5}  
 \begin{array}{lcl}
 &&{\sum}_{k\geq0}  \frac{\sigma_i}{10} \sum_{i}  (\| \triangle \bx^{\tk}  \|^2+\| \triangle\bx^{k+1}_i   \|^2 )\\
&\leq&{\sum}_{k\geq0}  ( \widetilde\L^{k}  -\widetilde\L^{k+1}) =
  \widetilde\L^{0}- f^* < +\infty. 
 \end{array}  
 \end{eqnarray*} 
It means 
 $  \| \triangle \bx^{\tk}   \|\rightarrow0$ and $\| \triangle \bx^{k+1}_i  \|\rightarrow0$, which by \eqref{opt-con-xk1-3} and \eqref{opt-con-xk1-4} derives  $ \| \triangle \bpi_i^{k+1}  \| \rightarrow0$.
% d) Direct verification brings out
%  \begin{eqnarray*} 
% \arraycolsep=1.0pt\def\arraystretch{1.5}  
% \begin{array}{lcl}
%\E \sum_{i} \| \triangle\blx^{k+1}_i   \|^2 &=&\E_{\Omega^\tk}{\sum}_{i\in \Omega^\tk}  \| \triangle\blx^{k+1}_i   \|^2  \\
% &=& \E_{\Omega^\tk}{\sum}_{i=1}^m \delta(i\in \Omega^\tk)  \| \triangle\blx^{k+1}_i   \|^2  \\
%  &=& {\sum}_{i=1}^m \P(i\in \Omega^\tk)  \| \triangle\blx^{k+1}_i   \|^2  \\
% &\overset{\eqref{iceadmm-sub3}  }{\leq}& {\sum}_{i=1}^m \P(i\in \Omega^\tk)  \| \triangle\bpi^{k+1}_i/\sigma_i   \|^2  \\
% \end{array}  
% \end{eqnarray*} 
% 
Finally, for  $i\in\tauk$, it follows from  \eqref{iceadmm-sub3}  that $\|\sigma _i \triangle\blx_i^{k+1}\| = \|  \triangle \bpi_i^{k+1} \| \to 0.$ 
For  $i\notin\tauk$, then $i\in\Omega^{\tau_{k_i+1}}$, where $k_i$ is defined the same as that in the proof of Lemma \ref{basic-observations} a). Based on \eqref{scheme-omega-2T}, we have $\tk -\tau_{k_i+1}\leq s_0$. As a consequence, 
 \begin{eqnarray} \label{gap-xi-xtau}
 \arraycolsep=1.0pt\def\arraystretch{1.5}  
 \begin{array}{lcl}
&&  \| \triangle\blx_i^{k+1}  \|^2 \overset{\eqref{par-invariance}}{=}\| \bx_i^{k_i+1}-\bx^{\tk}  \|^2 \\
&\leq& 2\| \triangle\blx_i^{k_i+1}\|^2+2\|\bx^{\tau_{k_i+1}}-\bx^{\tk}  \|^2 \\ 
&\overset{\eqref{iceadmm-sub3}}{=}& \frac{2}{\sigma_i^2}\| \triangle \bpi_i^{k_i+1} \|^2+2\|\sum_{t= \tau_{k_i+1}}^{\tau_{k+1}-1}(\bx^{ {t+1}}-\bx^{ t})  \|^2 \\ 
&\overset{\eqref{two-vecs}}{\leq}& %\frac{2}{\sigma_i^2}\|  \triangle \bpi_i^{k_i+1} \|^2+2({\tk}-\tau_{k_i+1})\sum_{t=\tau_{k_i+1}}^{{\tk}-1}\|\bx^{t}-\bx^{t+1}\|^2 \\ &{\leq}&
\frac{2}{\sigma_i^2}\| \triangle \bpi_i^{k_i+1} \|^2+2 s_0\sum_{t= \tau_{k_i+1}}^{\tau_{k+1}-1}\|\bx^{ {t+1}}-\bx^{ t}\|^2\\
& \to& 0,
 \end{array}  
 \end{eqnarray} 
where the last relationship is due to $\triangle \bpi_i^{k}\to$ and $\triangle\bx^{\tk}\to0$. The whole proof is finished.
\end{proof}

 \subsection{Proof of  Theorem \ref{global-obj-convergence-inexact}}   
{The sketch of proving results in Theorem \ref{global-obj-convergence-inexact} is as follows: The boundedness of $\{\bx^{\tk}\}$ can be ensured by Lemma \ref{L-bounded-decreasing-1} b) and the coerciveness of $f$. Since $\{\widetilde\L^k\}$ is non-increasing and bounded from below.   Therefore,  whole sequence $\{\widetilde\L^k\}$ converges, which by \eqref{limit-5-term-0-inexact} can show \eqref{L-local-convergence-limit-inexact}. Finally, to show  $\nabla F(W^{k+1}){\to} 0$ as $k{\to}\infty$, we only need to show $\nabla F(W^{\ell+1}) {\to} 0$ as $\ell(\in\K){\to}\infty$ due to \eqref{limit-5-term-0-inexact}, which can be guaranteed by conditions \eqref{opt-con-xk1-1} and \eqref{opt-con-xk1-2}.
}
 \begin{proof} a) By  Lemma \ref{L-bounded-decreasing-1}  that $\widetilde\L^1\geq f(\bx^{\tk})$ and $f$ being coercive,  we can claim the boundedness of sequence $\{\bx^{\tk}\}$ immediately.  This calls forth the boundedness of sequence $\{\bx_i^{k+1}\}$ as $ \triangle\blx_i^{k+1} \rightarrow0$ from \eqref{limit-5-term-0-inexact}, thereby delivering
  \begin{eqnarray*} 
\arraycolsep=0.0pt\def\arraystretch{1.5}
 \begin{array}{lcl}
 \|\bpi_i^{k+1}\|  &\overset{\eqref{opt-con-xk1-2}}{=}&   \|\bva_i^{k+1}-\bg_i^{k+1}\| \\
 &   \leq& \|\bva_i^{k+1}\| +  \|\bg_i^{k+1} - \bg_i^{0}\|+ \|\bg_i^{0}\|  \\
 & \overset{\eqref{opt-con-xk1-2},\eqref{Lip-r}}{\leq}&\sqrt{\epsilon_i^{k+1}}+ \alpha_{i} r_i \|  \bx^{k+1}_i-\bx^{0}_i\|+ \|\bg_i^{0}\|  <+\infty.  
     \end{array} 
 \end{eqnarray*}  
 This shows the boundedness of sequence $\{\bpi_i^{k+1}\}$. Overall,  sequence $\{(\bx^{\tk},W^{k+1},\Pi^{k+1})\}$ is bounded.

b) It follows from  Lemma \ref{L-bounded-decreasing-1} that $\{\widetilde\L^k\}$ is non-increasing and bounded from below.   Therefore,  whole sequence $\{\widetilde\L^k\}$ converges and $\widetilde\L^{k+1}  \to  \L^{k+1}$
due to $  \epsilon_i^{k+1}\rightarrow0$ in \eqref{limit-5-term-0-inexact}.  Again by  \eqref{limit-5-term-0-inexact} and  the boundedness of sequence $\{\bpi_i^{k+1}\}$, we can prove that 
\begin{eqnarray} \label{L-f-pi} 
 \arraycolsep=1.0pt\def\arraystretch{1.5}\begin{array}{lcl}
&& \L^{k+1}- F(W^{k+1}) \\
    &\overset{\eqref{Def-L}}{=}&  \sum_{i}  (  \langle \triangle \blx_i^{k+1}, \bpi_i^{k+1}\rangle +\frac{\sigma _i}{2}\|\triangle \blx_i^{k+1}\|^2 )  \to0.
    \end{array} 
 \end{eqnarray}     
 It follows from Mean Value Theory that
 $$f_i(\bx_i^{k+1})=f_i(\bx^{\tk})+\langle\triangle \blx^{k+1} , \nabla f_i(\bx_t)\rangle,$$ 
 where $\bx_t := (1-t)\bx^{\tk}+t\bx_i^{k+1}$ for some $t\in(0,1)$. Since  $\{\bx^{\tk},\bx_i^{k+1}\}$ is bounded, so is $\bx_t$. This calls  forth $f_i(\bx_i^{k+1})-f_i(\bx^{\tk})\to0$ due to  $\triangle \blx^{k+1}\to0$.  
Using this condition   enables us to obtain
\begin{eqnarray*}  
 \arraycolsep=1.0pt\def\arraystretch{1.5} \begin{array}{lll}
 \L^{k+1}-f(\bx^{\tk})
    &=&  \sum_{i}  ( \alpha_{i}  f_i(\bx_i^{k+1})- \alpha_{i} f_i(\bx^{\tk})\\
    &+&      \langle\triangle \blx^{k+1}, \bpi_i^{k+1}\rangle +\frac{\sigma _i}{2}\|\triangle \blx^{k+1}\|^2  )  \to  0.
    \end{array} 
 \end{eqnarray*}      
c) Let $\ell:=(\tau_{k+1}-1)k_0 \in\K$, then for any $k$, it  has
\begin{eqnarray}\label{tau-tau-s}
 \arraycolsep=1.0pt\def\arraystretch{1.5}
\begin{array}{rcll} 
\ell+1&=& (\tau_{k+1}-1)k_0+1 \leq k+1 \leq \tau_{k+1}k_0,\\ 
\tau_{\ell+1}&=&\lceil (\ell+1)/k_0 \rceil = \lceil  \tau_{k+1}-1-1/k_0 \rceil= \tau_{k+1}.
\end{array} \end{eqnarray}
 Since $\ell\in\K$, we have
\begin{eqnarray*}
 %\label{opt-con-xk1-1-ell}
  \arraycolsep=1.0pt\def\arraystretch{1.5}
\begin{array}{rcll}
&&\sum_{i}  \bpi_i^{\ell+1} \\
&\overset{\eqref{limit-5-term-0-inexact}}{\to}& \sum_{i}   ( \sigma_i(\triangle \blx_i^{\ell+1}-\triangle \bx_i^{\ell+1})-\triangle \bpi_i^{\ell+1}+ \bpi_i^{\ell+1} )\\
&=&\sum_{i}   ( \sigma_i(\bx_i^{\ell}-\bx^{\tau_{\ell+1}})+\bpi_i^{\ell})\overset{\eqref{opt-con-xk1-1}}{=} 0. 
\end{array} \end{eqnarray*} 
We note that sequence $\{\epsilon_i^{k+1}\}$ is non-increasing and thus obtain ${\epsilon_i^{\ell+1}} \leq{\epsilon_i^{k+1}}$ from \eqref{tau-tau-s}, thereby rendering that
 \begin{eqnarray*} %\label{limit-K-grad-i-exact}
 \arraycolsep=1.0pt\def\arraystretch{1.5}  
 \begin{array}{lcl}
&&\|\bpi_i^{k+1}-\bpi_i^{\ell+1}\|^2\\
 &\overset{\eqref{opt-con-xk1-2}}{=}&   \| \bva_i^{k+1}-\bva_i^{\ell+1}-\bg_i^{k+1} +\bg_i^{\ell+1} \|^2  \\
    &\overset{\eqref{opt-con-xk1-2},\eqref{Lip-r}}{\leq}&     3{\epsilon_i^{k+1}}+ 3{\epsilon_i^{\ell+1}} + 3\alpha_{i}^2r_i^2 \| \bx^{k+1}_i-\bx^{\ell+1}_i\|^2\\
     &\overset{\eqref{tau-tau-s}}{=}&   6{\epsilon_i^{k+1}}  + 3\alpha_{i}^2r_i^2 \| \bx^{k+1}_i-\bx^{\tk} + \bx^{\tau_{\ell+1}}- \bx^{\ell+1}_i\|^2\\
     &{\leq}&   6{\epsilon_i^{k+1}}  + 6\alpha_{i}^2r_i^2 (\| \triangle \blx^{k+1}_i\|^2 + \|\triangle\blx^{{\ell+1}}_i\|^2) \overset{\eqref{limit-5-term-0-inexact}}{\to}0.
   \end{array}
  \end{eqnarray*} 
Using the above two conditions immediately derives that
  \begin{eqnarray} \label{limit-K-grad-pi-exact}
  \begin{array}{llllllll}
   \sum_{i}  \bpi_i^{k+1} \to 0.
    \end{array}\end{eqnarray} 
Taking the limit on both sides of \eqref{opt-con-xk1-2} gives us
 \begin{eqnarray} \label{limit-K-grad-i-exact}
\arraycolsep=1.0pt\def\arraystretch{1.5}
  \begin{array}{lcl}  
   \nabla F(W^{k+1}) &=& \sum_{i}  \bg_i^{k+1}\\
   &\overset{\eqref{limit-K-grad-pi-exact}}{\to}&\sum_{i}  (\bg_i^{k+1} +\bpi_i^{k+1})\\
   &\overset{\eqref{opt-con-xk1-2}}{=}&\sum_{i}  \bva_i^{k+1} \overset{\eqref{opt-con-xk1-3}}{\to}0,
   \end{array}
  \end{eqnarray} 
which together with $\triangle\blx_i^{k+1}\rightarrow0$ and the gradient Lipschitz continuity yields that $ \nabla f (\bx^{\tk})=\sum_{i}  \alpha_{i} \nabla f_i(\bx^{\tk})\to 0.$  This completes the whole proof.
\end{proof}

 \subsection{Proof of  Theorem \ref{global-convergence-inexact}}   

 \begin{proof} 
a)  Let $(\bx^{\infty},W^{\infty},\Pi^{\infty})$ be any accumulating point of the sequence, 
 it follows from \eqref{opt-con-xk1-2} and \eqref{opt-con-xk1-3} that
  \begin{eqnarray*}
  \begin{array}{l}
0= \alpha_{i} \nabla f_i(\bx^{\infty}_i) + \bpi_i^{\infty}.
   \end{array}
  \end{eqnarray*} 
By  $ (\bx_i^{k+1}-\bx^{\tk})\rightarrow0$ and \eqref{limit-K-grad-pi-exact}, we have 
 \begin{eqnarray*}
  \begin{array}{l}
0=\bx^{\infty}_i- \bx^{\infty},~~ 0=\sum_{i} \bpi^{\infty}_i.
   \end{array}
  \end{eqnarray*} 
  Therefore, recalling \eqref{opt-con-FL-opt-ver1},  $(\bx^{\infty},W^{\infty},\Pi^{\infty})$  is a stationary point of   (\ref{FL-opt-ver1}) and $\bx^{\infty}$ is a stationary point of   (\ref{FL-opt}). 
  
b) It follows from \cite[Lemma 4.10]{more1983computing},  $\triangle \bx^{\tk} \rightarrow0$ and $\bx^{\infty}$ being isolated that  the whole sequence, $\{\bx^{\tk}\}$ converges to $\bx^{\infty}$, which by $\triangle\blx_i^{k+1}\rightarrow0$ implies that $\{W^{k+1}\}$ also converges to $W^{\infty}$. Finally, this together with \eqref{opt-con-xk1-2} and \eqref{opt-con-xk1-3}   results in the convergence of $\{\Pi^{k+1}\}$.
  \end{proof}

 \subsection{Proof of  Corollary \ref{L-global-convergence}}   
\begin{proof} a) The convexity of $f$ and the optimality of $\bx^*$ yields
    \begin{eqnarray}  \label{convexity-optimality}
   \begin{array}{lll}
f(\bx^{\tau_k}) \geq  f(\bx^{*}) \geq  f(\bx^{\tau_k}) + \langle \nabla f(\bx^{\tau_k}), \bx^{*}{-}\bx^{\tau_k} \rangle.~ 
    \end{array} 
 \end{eqnarray} 
 Theorem \ref{global-obj-convergence-inexact} ii) states that   
   \begin{eqnarray*}  
   \begin{array}{lll}
 {\lim}_{k \rightarrow \infty} \nabla F(W^{k})  ={\lim}_{k \rightarrow\infty} \nabla f(\bx^{\tau_k}) =0.
    \end{array} 
 \end{eqnarray*} 
 Using this and the boundedness of $\{\bx^{\tau_k}\}$ from Theorem \ref{global-convergence-inexact}, we take the limit of both sides of \eqref{convexity-optimality} to derive that  $ f(\bx^{\tau_k})\rightarrow f(\bx^{*})$, which recalling  Theorem \ref{global-obj-convergence-inexact} i) yields \eqref{L-global-convergence-limit}. 
 
b) The conclusion follows from  Theorem \ref{global-convergence-inexact} ii) and the fact that  the stationary points are equivalent to optimal solutions if $f$ is convex.
 
c)   The strong convexity of $f$ means that
 there is a positive constance $\nu$ such that 
 \begin{eqnarray*}  
   \arraycolsep=1.0pt\def\arraystretch{1.5}
  \begin{array}{llll}
  f(\bx^{\tau_k}) -f( \bx^*)
  &\geq&  \langle \nabla f( \bx^*),\bx^{\tau_k}-\bx^*\rangle + \frac{\nu}{2}\|\bx^{\tau_k}-\bx^*\|^2\\
  &=&  \frac{\nu}{2}\|\bx^{\tau_k}-\bx^*\|^2, 
    \end{array} 
 \end{eqnarray*}
where the equality is due to \eqref{grad-x-*=0}. Taking limit of both sides of the above inequality  shows $\bx^{\tau_k}\rightarrow\bx^*$ since $ f(\bx^{\tau_k}) \rightarrow f( \bx^*)$. This together with \eqref{limit-5-term-0-inexact} yields $ \bx_i^k\rightarrow\bx^*$. Finally, $ \bpi_i^k\rightarrow\bpi_i^*$ because of   
 \begin{eqnarray*}  
   \arraycolsep=1.0pt\def\arraystretch{1.5}
  \begin{array}{lcl}\|\bpi_i^{k}-\bpi_i^{*}\|^2&\overset{\eqref{opt-con-FL-opt-ver1},\eqref{opt-con-xk1-2}}{=}&\|\bva_i^{k} + \bg_i^{k} - \alpha_{i} \nabla f_i(\bx^*)\|^2\\
  & \overset{\eqref{Lip-r}}{\leq}&  2\alpha_{i}^2r_i^2\| \bx^{k}_i-\bx^*\|^2+ 2{\epsilon_i^k}\to0,
   \end{array} 
 \end{eqnarray*} 
 displaying the desired result.
\end{proof}

 \subsection{Proof of  Theorem \ref{complexity-thorem-gradient-inexact}}   
 {The proof focuses on $k\in\K$ and aims at estimating term $\|\nabla f(\bx^{\tau_{k +1}})\|^2 $ by decomposing it as 
\begin{eqnarray} \label{decom-grad-f-tau}
   \arraycolsep=1pt\def\arraystretch{1.5}
\begin{array}{rcll}
&& \|\nabla f(\bx^{\tau_{k +1}})\|^2 \\&\leq& 3\|\nabla f(\bx^{\tau_{k +1}})-\sum_{i} \bg_i^{k +1}\|^2\\
&+& 3\|\sum_{i} \bg_i^{k +1}-\sum_{i} \bg_i^{k }\|^2+3\|\sum_{i} \bg_i^{k }\|^2\\
&{\leq}& m\sum_{i} 3\alpha_i^2r_i^2(\| \triangle \blx_i^{{k +1}}\|^2+\| \triangle \bx_i^{{k +1}}\|^2)+ 3\|\sum_{i} \bg_i^{k }\|^2\\
&{\leq}& m\sum_{i} \frac{\sigma_i^2}{3} (\| \triangle \blx_i^{{k +1}}\|^2+   \| \triangle \bx_i^{{k +1}}\|^2)+3\|\sum_{i} \bg_i^{k }\|^2,
\end{array}\end{eqnarray}
where the last two inequalities used the gradient Lipschitz continuity and \eqref{sigma-a-r}.    Then we estimate each term on the right-hand side. The detailed proof is given as follows.
 }
{
\begin{proof}\underline{Estimate $\sum_{i} \|\sigma _i\triangle \bx_i^{k+1} \|^2$.}  Recalling Lemma \ref{lemma-decreasing-1},
\begin{eqnarray*}   
 \arraycolsep=1.0pt\def\arraystretch{1.5}
 \begin{array}{lll}
 \sum_{i} \frac{\sigma _i}{10}(\|\triangle \bx^{\tk} \|^2      +\|\triangle \bx_i^{k+1} \|^2) \leq \widetilde\L^k- \widetilde\L^{k+1}, 
    \end{array}  
 \end{eqnarray*}  
 which  by letting $\overline\sigma:=\max_{i\in[m]}\sigma_i$ results in
 \begin{eqnarray}\label{gap-w-w-L-L}   
 \arraycolsep=1.0pt\def\arraystretch{1.5}
 \begin{array}{lll}
 && \max\{ \sum_{i}  \|\sigma _i \triangle \bx^{\tk} \|^2, \sum_{i}    \|\sigma _i\triangle \bx_i^{k+1} \|^2\}\\
 &\leq& 10 \overline\sigma  \sum_{i} \frac{\sigma _i}{10}(\|\triangle \bx^{\tk} \|^2      +\|\triangle \bx_i^{k+1} \|^2)\\
 &\leq&10 \overline\sigma( \widetilde\L^k- \widetilde\L^{k+1}).
% &\leq&10 \overline\sigma( \widetilde\L^{k-s_0k_0}- \widetilde\L^{k+1}). 
    \end{array}  
 \end{eqnarray} 
\underline{Estimate $\sum_{i}\|\sigma_i \triangle\blx_i^{k +1}  \|^2$.}  For any $i\in\tauk$, by the third inequality in \eqref{gap-pi-k-eps}, we have 
\begin{eqnarray} \label{opt-con-xk1-3-1}
 \arraycolsep=0.5pt\def\arraystretch{1.5}
\begin{array}{lcl}
\|\triangle   \bpi_i^{k+1}\|^2  
&\overset{\eqref{gap-pi-k-eps}}{\leq}&   \frac{6\alpha_{i}^2r_i^2}{5}  \|\triangle\bx_i^{k+1} \|^2+12(\epsilon_i^{k+1}+\epsilon_i^{k})\\
&\overset{(\ref{sigma-a-r})}{\leq}&\frac{2\sigma_i^2}{15} \|\triangle\bx_i^{k+1} \|^2+12(\epsilon_i^{k+1}+\epsilon_i^{k})\\
&\overset{(\ref{iceadmm-sub20})}{\leq}&\frac{2\sigma_i^2}{15} \|\triangle\bx_i^{k+1} \|^2+ 24\epsilon_i^{k}.
   \end{array}
  \end{eqnarray}
Then it follows  
\begin{eqnarray} \label{gap-xi-xtau-109}
 \arraycolsep=1.0pt\def\arraystretch{1.5}  
 \begin{array}{lcl}
 \|\sigma _i \triangle\blx_i^{k +1}\|^2 &\overset{ \eqref{iceadmm-sub3}}{=}& \|  \triangle \bpi_i^{k +1} \|^2\\
&\overset{\eqref{opt-con-xk1-3-1}}{\leq}&\frac{2\sigma_i^2}{15} \|\triangle\bx_i^{k +1} \|^2 +{24}\epsilon_i^{k }.
 \end{array}  
 \end{eqnarray} 
For  $i\notin\tauk$, let $k_i$ be defined similarly to that in the proof of Lemma \ref{basic-observations} a) at step $k $. Then $i\in\Omega^{\tau_{k_i+1}}$ and \eqref{opt-con-xk1-3-1} holds for $k_i$, which allows us to obtain
\begin{eqnarray} \label{gap-xi-xtau-11}
 \arraycolsep=1.0pt\def\arraystretch{1.5}  
 \begin{array}{lcl}
  \|\sigma_i \triangle\blx_i^{k +1}  \|^2 
  &\overset{\eqref{gap-xi-xtau}}{\leq} & 2\| \triangle \bpi_i^{k_i+1} \|^2\\
&+&2s_0\sigma_i^2\sum_{t=\tau_{k_i+1}}^{\tau_{k+1}-1 }\| \bx^{ {t+1}}-\bx^{ {t}}\|^2\\
&\overset{\eqref{opt-con-xk1-3-1}}{\leq}& \frac{4\sigma_i^2}{15} \|\triangle\bx_i^{k_i+1} \|^2 + {48} \epsilon_i^{k_i}\\
&+&2s_0\sigma_i^2\sum_{t=\tau_{k_i+1}}^{\tau_{k+1} -1}\| \bx^{ {t+1}}-\bx^{ {t}}\|^2.
 \end{array}  
 \end{eqnarray} 
Based on \eqref{scheme-omega-2T}, we have $\tau_{k +1} -\tau_{k_i+1}\leq s_0$ and thus $k\geq k_i \geq k -(s_0-1)k_0-1\geq k -s_0k_0,$ which by the non-increasing properties of $\{\widetilde\L^{k}\}$ and $\{\epsilon_i^k\}$  imply that
\begin{eqnarray} \label{Lki-Lks0}
 \arraycolsep=1.0pt\def\arraystretch{1.5}  
 \begin{array}{lcl}
\widetilde\L^{k} \leq \widetilde\L^{k_i}  \leq  \widetilde\L^{k -s_0k_0},~~~ \epsilon_i^{k}\overset{\eqref{iceadmm-sub5}}{=}\epsilon_i^{k_i},  i\notin\tauk.
 \end{array}  
 \end{eqnarray} 
By (\ref{gap-xi-xtau-109}) and (\ref{gap-xi-xtau-11}) we can obtain
\begin{eqnarray} \label{gap-xi-xtau-112}
 \arraycolsep=0.5pt\def\arraystretch{1.75}  
 \begin{array}{lcl}
 &&  \sum_{i}\|\sigma_i \triangle\blx_i^{k +1}  \|^2 \\
 &=&
  \sum_{i\in\tauk}\|\sigma_i \triangle\blx_i^{k +1}  \|^2 + \sum_{i\notin\tauk}\|\sigma_i \triangle\blx_i^{k +1}  \|^2\\
  & {\leq}&
  \sum_{i\in\tauk}(\frac{2\sigma_i^2}{15} \|\triangle\bx_i^{k +1} \|^2 +{24}\epsilon_i^{k })\\
  &+& \sum_{i\notin\tauk} (\frac{4\sigma_i^2}{15} \|\triangle\bx_i^{k_i+1} \|^2 + {48} \epsilon_i^{k_i}\\
&+& 2s_0\sigma_i^2\sum_{t=\tau_{k_i+1}}^{\tau_{k+1} -1}\| \bx^{ {t+1}}-\bx^{ {t}}\|^2 )\\ 
 & {\leq}&
  \sum_{i}(\frac{2\sigma_i^2}{15} \|\triangle\bx_i^{k +1} \|^2 + \frac{4\sigma_i^2}{15} \|\triangle\bx_i^{k_i+1} \|^2 )\\
  &+& \sum_{i}  (2s_0\sigma_i^2\sum_{t=\tau_{k_i+1}}^{\tau_{k+1} -1}\| \bx^{ {t+1}}-\bx^{ {t}}\|^2+ {48} \epsilon_i^{k}). 
%&\overset{\eqref{gap-w-w-L-L}}{\leq} &\frac{4\overline{\sigma}}{3}(\widetilde\L^k- \widetilde\L^{k+1})+ \frac{8\overline{\sigma}}{3} (\widetilde\L^{k_i}- \widetilde\L^{k_i+1}) \\ 
%&+&  20 s_0\overline{\sigma}\sum_{t=k_i+1}^{k }  (\widetilde\L^t- \widetilde\L^{t+1}) +  \sum_{i}  {48} \epsilon_i^{k}  \\
%& {\leq} &\frac{4\overline{\sigma}}{3}(\widetilde\L^k{-} \widetilde\L^{k+1}){+}  \sum_{i}   {48} \epsilon_i^{k} {+} 20 s_0\overline{\sigma} (\widetilde\L^{k_i}{- }\widetilde\L^{k+1}) \\ 
%  & \overset{\eqref{Lki-Lks0}}{\leq} &\frac{64s_0\overline{\sigma}}{3}   (\widetilde\L^{ k -s_0k_0 }- \widetilde\L^{k+1}) + \sum_{i}  48 \epsilon_i^{ k}.
 \end{array}  
 \end{eqnarray} 
\underline{Estimate $\|\sum_{i} \bg_i^{k }\|^2 $.} For any $k \in\K$, we have
\begin{eqnarray*} 
   \arraycolsep=1.0pt\def\arraystretch{1.5}
\begin{array}{rcll} 
\sum_{i} \bg_i^{k } &\overset{\eqref{opt-con-xk1-2}}{=}& \sum_{i} (\bva_i^{k }-  \bpi_i^{k })\\
&\overset{\eqref{opt-con-xk1-1}}{=}& \sum_{i} (\bva_i^{k } + \sigma_i(\bx_i^{k }-\bx^{\tau_{k +1}}))\\
&{=}& \sum_{i} (\bva_i^{k } + \sigma_i( \triangle \blx_i^{k +1}- \triangle \bx_i^{k +1})),
\end{array} \end{eqnarray*} 
which by $\|\bva_i^{k}\|^2\leq \epsilon_i^k $ from \eqref{opt-con-xk1-2} leads to
\begin{eqnarray} \label{sum-g-i}
   \arraycolsep=1.0pt\def\arraystretch{1.5}
\begin{array}{rcll} 
\|\sum_{i} \bg_i^{k }\|^2 
%&\overset{\eqref{two-vecs}}{\leq}&  3\|\sum_{i}\bva_i^{t}\|^2 \\
%&+& 3\| \sum_{i} \sigma_i \triangle \blx_i^{t+1}\|\\
%& +&3\| \sum_{i} \sigma_i\triangle \bx^{t+1}\|\\
 &\overset{\eqref{two-vecs}}{\leq}&   3 m \sum_{i} \epsilon_i^{k }\\
&+& 3m\sum_{i} \sigma_i^2\|  \triangle \blx_i^{k +1}\|^2\\
& +&3m\sum_{i} \sigma_i^2\| \triangle \bx_i^{k +1}\|^2.
\end{array} \end{eqnarray} 
Now, combining \eqref{decom-grad-f-tau} and  \eqref{sum-g-i} yields that 
\begin{eqnarray} \label{error-grad-f-w}
   \arraycolsep=0pt\def\arraystretch{1.5}
\begin{array}{rcll}
&&\|\nabla f(\bx^{\tau_{k +1}})\|^2\\ 
%&\overset{\eqref{decom-grad-f-tau}}{\leq}&  m\sum_{i} \frac{\sigma_i^2}{3} (\| \triangle \blx_i^{{k +1}}\|^2{+}   \| \triangle \bx_i^{{k +1}}\|^2){+}3\|\sum_{i} \bg_i^{k }\|^2\\
&{\leq}& \sum_{i} \frac{10m\sigma_i^2}{3} (\| \triangle \blx_i^{{k +1}}\|^2{+} \| \triangle \bx_i^{{k +1}}\|^2){+}9 m \sum_{i} \epsilon_i^{k }\\
&\overset{\eqref{gap-xi-xtau-112}}{\leq}&
%\frac{10m}{3}\sum_{i}(\frac{2\sigma_i^2}{15} \|\triangle\bx_i^{k +1} \|^2 + \frac{4\sigma_i^2}{15} \|\triangle\bx_i^{k_i+1} \|^2 )\\
%  &+& \frac{10m}{3} \sum_{i}  (2s_0\sigma_i^2\sum_{t=\tau_{k_i+1}}^{\tau_{k+1} -1}\| \bx^{ {t+1}}-\bx^{ {t}}\|^2+ {48} \epsilon_i^{k})\\
%&+&  \frac{10m}{3}\sum_{i} \sigma_i^2 \| \triangle \bx_i^{{k +1}}\|^2+9 m \sum_{i} \epsilon_i^{k }\\
 \frac{34m}{9}\sum_{i}  \|\sigma_i\triangle\bx_i^{k +1} \|^2 \\
&+& \frac{8m}{9}\sum_{i} \|\sigma_i\triangle\bx_i^{k_i+1} \|^2+ {169m}\sum_{i}   \epsilon_i^{ k}\\
&+&
\frac{20ms_0}{3}\sum_{t=\tau_{k_i+1}}^{\tau_{k+1} -1} \sum_{i} \|\sigma_i( \bx^{ {t+1}}-\bx^{ {t}})\|^2\\
&\overset{\eqref{gap-w-w-L-L}}{\leq}& \frac{340m\overline{\sigma}}{9}(\widetilde\L^k- \widetilde\L^{k+1})+\frac{80m\overline{\sigma}}{9}(\widetilde\L^{k_i}- \widetilde\L^{k_i+1})\\
&+&  {169m}\sum_{i}   \epsilon_i^{ k} +
\frac{200ms_0\overline{\sigma}}{3}\sum_{t=k_i+1}^{k}(\widetilde\L^{t}- \widetilde\L^{t+1})\\
&{\leq}& \frac{340m\overline{\sigma}}{9}(\widetilde\L^k- \widetilde\L^{k+1})+{169m}\sum_{i}   \epsilon_i^{ k}\\
&+& 
\frac{200ms_0\overline{\sigma}}{3}(\widetilde\L^{k_i}- \widetilde\L^{k+1})\\
&\overset{\eqref{Lki-Lks0} }{\leq}&  \frac{940m s_0\overline{\sigma}}{9}(\widetilde\L^{ k -s_0k_0}- \widetilde\L^{k+1})+    {169m}\sum_{i}   \epsilon_i^{ k}.
\end{array} \end{eqnarray}
%Now let $k=s_0,2s_0,3s_0\cdots,ts_0$. 
The non-increasing property of $\{\widetilde\L^{k}\}$ suffices to
\begin{eqnarray} \label{L-2s-L-2s}
   \arraycolsep=1.0pt\def\arraystretch{1.5}
\begin{array}{rcll}
\widetilde\L^{(p+1)k_0}-\widetilde\L^{pk_0+1}  &\leq& 0, ~~\forall p\geq 0.
%\widetilde\L^{(p+1) s_0k_0} -\widetilde\L^{pk_0+1} &\leq& 0,~~\forall p\geq 0.
\end{array} \end{eqnarray}
We note that $\{\epsilon_i^k\}$ is non-increasing. Moreover, under Scheme \ref{ass-omega},  for every  $s_0k_0$ steps, each client $i\in[m]$ is chosen al least once to update their parameters by \eqref{iceadmm-sub20}-\eqref{iceadmm-sub4}, which means that
$\epsilon_i^{(a+1)s_0k_0}\leq \nu_i \epsilon_i^{as_0k_0}$ for all $i\in[m]$ and $a=1,2,\cdots$, thereby leading to 
\begin{eqnarray} \label{sum-eps-i}
   \arraycolsep=1pt\def\arraystretch{1.75}
\begin{array}{rcll} 
 \sum_{p=s_0}^{t} \epsilon_i^{pk_0}  &\leq&\sum_{a=1}^{\lceil  {t}/{s_0}\rceil } s_0 \epsilon_i^{as_0k_0}\\
 &\leq&   \frac{s_0\epsilon_i^{s_0k_0} (1-\nu_i^{\lceil  {t}/{s_0}\rceil})}{1-\nu_i} \leq \frac{s_0\epsilon_i}{1-\nu_i}.
\end{array} \end{eqnarray}
Finally, using (\ref{L-2s-L-2s}) and (\ref{sum-eps-i}) enables us to derive
\begin{eqnarray*} 
   \arraycolsep=1.0pt\def\arraystretch{1.5}
\begin{array}{rcll}
&&(t-s_0+1)\min_{p=s_0,s_0+1,\cdots,t}\|\nabla f(\bx^{\tau_{pk_0 +1}})\|^2\\
 &\leq&  \sum_{p=s_0}^{t} \|\nabla f(\bx^{\tau_{pk_0 +1}})\|^2\\
&\overset{\eqref{error-grad-f-w}}{\leq}&  \sum_{p=s_0}^{t}  \frac{940m  s_0\overline{\sigma}}{9}(\widetilde\L^{(p-s_0)k_0}- \widetilde\L^{pk_0+1})\\ 
&+& \sum_{p=s_0}^{t} {169m}\sum_{i}   \epsilon_i^{pk_0} \\
&{\leq}&    \sum_{i} \frac{169m s_0\epsilon_i}{ 1-\nu_i } +   \frac{940m s_0\overline{\sigma}}{9}(\widetilde\L^{0}+\widetilde\L^{k_0}+\cdots + \widetilde\L^{s_0k_0})\\
&-&   \frac{940m s_0\overline{\sigma}}{9}(\widetilde\L^{(t-s_0)k_0+1}+\widetilde\L^{(t-s_0+1)k_0+1}+\cdots+\widetilde\L^{tk_0+1})\\
   &{\leq}&    \frac{940m\overline{\sigma} s_0(s_0+1)}{9}   (\widetilde\L^{0} - f^*)+\sum_{i} \frac{169ms_0\epsilon_i}{1-\nu_i}= c, 
\end{array} \end{eqnarray*}
where the last inequality is due to $ \widetilde\L^0\geq\widetilde\L^k \geq f^*$ for any $ k\geq1$ from Lemma \ref{L-bounded-decreasing-1} b). Now by letting $k=tk_0$, we have
\begin{eqnarray*} 
   \arraycolsep=1.0pt\def\arraystretch{1.5}
\begin{array}{rcll}
&&\min_{s=1,3,\cdots,k}\|\nabla f(\bx^{\tau_{s +1}})\|^2 \\
&\leq &\min_{p=s_0,s_0+1,\cdots,t}\|\nabla f(\bx^{\tau_{pk_0 +1}})\|^2\\
 &\leq&  \frac{c}{t-s_0+1} \leq \frac{c k_0}{k -s_0k_0}, 
\end{array} \end{eqnarray*}
showing the desired result.
\end{proof}
}

\end{document}